\documentclass[12pt,reqno]{amsart}
\usepackage{amsmath, amssymb, graphicx, mathrsfs, dsfont, hyperref}
\usepackage{epsfig}
\usepackage[usenames,dvipsnames,svgnames]{xcolor}
\usepackage{amsmath,amssymb,url,setspace,mathtools}
\usepackage[all]{xy}

\newtheorem{theorem}{Theorem}[section]
\newtheorem{definition}[theorem]{Definition}
\newtheorem{lemma}[theorem]{Lemma}
\newtheorem{proposition}[theorem]{Proposition}
\newtheorem{corollary}[theorem]{Corollary}

\newtheorem{remark}[theorem]{Remark}

\newtheorem{example}[theorem]{Example}

\numberwithin{equation}{section}

\newcommand{\e}{\varepsilon}

\newcommand{\BB}{{\mathbb B}}
\newcommand{\HH}{{\mathcal H}}
\newcommand{\K}{{\mathcal K}}
\newcommand{\LL}{{\mathcal L}}
\newcommand{\M}{{\mathcal M}}

\newcommand{\C}{\mathcal C}
\newcommand{\F}{\mathbb F}
\newcommand{\E}{\mathbb E}
\newcommand{\id}{\operatorname{id}}
\newcommand{\md}{\operatorname{md}}
\newcommand{\cb}{\operatorname{cb}}
\newcommand{\op}{\operatorname{op}}
\newcommand{\ev}{\operatorname{ev_1}}

\newcommand{\CE}{C^*\langle E\rangle}
\newcommand{\CF}{C^*\F_\infty}
\newcommand{\OS}{\mathcal{OS}_n}

\allowdisplaybreaks

\begin{document}

\title[]{Operator-valued Kirchberg Theory}

\author[]{Jian Liang and Sepideh Rezvani}

\address{Department of Mathematics, University of Illinois, Urbana,
IL 61801, U.S.A.} \email{liang41\@@illinois.edu,rezvani2\@@illinois.edu}
\date{\today}


\begin{abstract}
In this paper, we will follow Kirchberg's categorical perspective to establish new notions of WEP and QWEP relative to 
a C$^*$-algebra, and 
develop similar properties as in the classical WEP and QWEP. Also we will show some examples of relative 
WEP and QWEP to illustrate the relations with the classical cases. Finally we will apply our notions to recent results 
on C$^*$-norms.   
\end{abstract}

\maketitle

\section{Introduction}

In this paper, we investigate a new notion of operator-valued WEP and QWEP. Let us recall that 
the weak expectation property (abbreviated as WEP) was introduced by E. Christopher Lance
in his paper \cite{La2} of 1973, as a generalization of nuclearity of C$^*$-algebras. 
In 1993, Eberhard Kirchberg \cite{Ki1} revealed
remarkable connections between tensor products of C$^*$-algebras and Lance's weak expectation property. 
He defined the notion of QWEP as a quotient of a C$^*$-algebra with
the WEP, and formulated the famous QWEP conjecture that all C$^*$-algebras are QWEP. He showed a
vast amount of equivalences between various open problems in operator algebras. In particular, 
he showed that the QWEP conjecture is
equivalent to an affirmative answer to the Connes' Embedding Problem.

The motivation of our research is to generalize the notion of WEP and QWEP in the setting of Hilbert C$^*$-
modules, in which the inner product of a
Hilbert space is replaced by a C$^*$-valued inner product. Hilbert C$^*$-modules were first
introduced in the work of Irving Kaplansky in 1953 \cite{Kap53}, which developed the theory for
commutative, unital algebras. In the 1970s the theory was extended to noncommutative C$^*$-
algebras independently by William Lindall Paschke \cite{Pas73} and Marc Rieffel\cite{Rie74}. The latter 
used Hilbert C$^*$-modules to construct a theory of induced representations of C$^*$-algebras. 
Hilbert C$^*$-modules are crucial to Kasparov's formulation of KK-theory \cite{Kas80}, and
provide the right framework to extend the notion of Morita equivalence to C*-algebras \cite{Rie82}.
They can be viewed as the generalization of vector bundles to noncommutative C*-algebras and as
such play an important role in noncommutative geometry, notably in C$^*$-algebraic quantum group
theory and groupoid C$^*$-algebras.

Another motivation of our research is from the relation with amenable correspondence. The
notion of correspondence of two von Neumann algebras has been introduced by Alain Connes
\cite{CJ}, as a very useful tool for the study of type II$_1$ factors. Later Sorin Popa systematically developed this point of view to get some new insight in this area \cite{Pop}. Among many
interesting results and remarks, he discussed Connes' classical work on the injective II$_1$ factor in
the framework of correspondences, and he defined and studied a natural notion of amenability for
a finite von Neumann algebra relative to a von Neumann subalgebra using conditional expectations. As Lance was 
inspired by Tomiyama's work on conditional expectations, we are interested in weak conditional 
expectations relative to a C$^*$-algebra.

The main results of this paper are inspired by Kirchberg's seminal work on non-semisplit extensions. 
In Section 3, we define two notions of WEP relative to a C$^*$-algebra $D$. Let $E_D$ be a Hilbert $D$-module, 
and $\mathcal{L}(E_D)$ be the C$^*$-algebra of bounded adjointable linear operators on $E_D$. Also let $E_{D^{**}}$ be the weakly closed 
Hilbert $D^{**}$-module, and $\mathcal{L}^w(E_{D^{**}})$ be the von Neumann algebra of bounded adjointable linear operators on $E_{D^{**}}$. 
We say that a C$^*$-algebra $A$ has the $D$WEP$_1$ if it is relatively weakly injective in $\mathcal{L}(E_D)$, 
i.e. for a faithful representation $A \subset \mathcal{L}(E_D)$, there exists a ucp map $\mathcal{L}(E_D) \to A^{**}$, 
which preserves the identity on $A$. Respectively we define the $D$WEP$_2$ to be the relatively weak injectivity in 
$\mathcal{L}(E_{D^{**}})$. We show that $D$WEP$_1$ implies $D$WEP$_2$, but the converse is not true. 
After investigating some basic properties, we establish a tensor product characterization of $D$WEP. Let 
$max^D_1$ be the tensor norm on $A\otimes C^*\F_{\infty}$ induced from the inclusion $A\otimes C^*\F_{\infty} \subseteq \mathcal{L}(E^u_D)\otimes_{max}C^*\F_{\infty}$ 
for some universal Hilbert $D$-module $E^u_D$ and $A\subset \mathcal{L}(E^u_D) $. Then a
 C$^*$-algebra $A$ has the $D$WEP$_1$, if and only if
\[
 A\underset{max^D_1}\otimes C^*\F_{\infty} = A \underset{max}\otimes C^*\F_{\infty}.
\] 
We have the similar result for $D$WEP$_2$ with respect to some universal weakly closed $D^{**}$-module $E^u_{D^{**}}$. 

In Section 4, we define two notions of relative QWEP, derived from relative WEP. After developing basic properties of relative QWEP, we show that the two notions are equivalent, unlike the case in the relative 
WEP. Similarly, we establish a tensor product characterization of relative QWEP. In Section 5, we investigate some properties of WEP and QWEP relative to some special classes of C$^*$-algebras, and illustrate the 
relations with classical results in the WEP and QWEP theory. 

In Section 6, we discuss some application of our tensor product characterization result for $D$WEP from Section 3 in the setting 
of C$^*$-norms. As we know, the algebraic tensor product $A\otimes B$ of two C$^*$-algebras may admit distinct norms, for 
instance, the minimal and maximal norms. A C$^*$-algeba $A$ such that $\|\cdot\|_{\min}=\|\cdot\|_{\max}$ on $A\otimes B$ for any other 
C$^*$-algebra $B$ is called nuclear. In particular, Simon Wassermann \cite{Was76} shows that $\BB(\HH)$ is not nuclear, and 
later Gilles Pisier and Marius Junge \cite{JP95} show that $\|\cdot\|_{\min}\neq\|\cdot\|_{\max}$ on $\BB(\HH)\otimes \BB(\HH)$. Recently, 
Pisier and Ozawa \cite{OP14} showed that there is at least a continuum of different C$^*$-norms on 
$\BB(\ell_2)\otimes \BB(\ell_2)$. In Section 6, we adopt the adea in their paper to construct a new C$^*$-norm 
on $A\otimes B$ by using the notion of $D$WEP and the $max^D_1$ norm which we constructed in Section 3, and provide the conditions which make it
 neither min nor max norm, and distinct from the continuum norms constructed by Pisier and Ozawa. 
We also give a concrete example, satisfying the conditions and hence with four distinct tensor norms. 
These conditions will give us a new way to distinguish norms on C$^*$-algebras. 

We would like to thank Marius Junge, for having extensive inspiring discussions on this topic.

\section{Preliminaries}

\subsection{WEP and QWEP}

The notion of WEP is from Lance \cite{La2}, and it is inspired by Tomiyama's extensive work on conditional expectations. 
Kirchberg in \cite{Ki1} raises the famous QWEP conjecture and establishes its several equivalences. 
Here we list some useful results for readers' convenience. Most of the results and proofs 
can be found in Ozawa's survey paper \cite{Oz}. 

\begin{definition}
 Let $A$ be a unital C$^*$-subalgebra of a unital C$^*$-algebra $B$. We say $A$ is
 relatively weakly injective (short as r.w.i.) in $B$, if there 
is a ucp map $\varphi: B \to A^{**}$ such that $\varphi|_A = \id_{A}$.
\end{definition}

For von Neumann algebras $M\subset N$, the relative weak injectivity is equivalent to the existence of 
a (non-normal) conditional expectation from $N$ to $M$.

 We say a C$^*$-algebra $A$ has the weak expectation property (short as WEP), 
if it is relatively weakly injective in $\mathbb{B}(\mathcal{H})$ for a faithful representation 
$A\subset \BB(\HH)$.

Since $\BB(\HH)$ is injective, the notion of WEP does not depend on the choice of a faithful representation of $A$. 
We say a C$^*$-algebra is QWEP if it is a quotient of a C$^*$-algebra with the WEP. The QWEP conjecture raised by Kirchberg 
in \cite{Ki1} states that all C$^*$-algebras are QWEP. 

From the definition of \emph{r.w.i.}, it is easy to see the following transitivity property.

\begin{lemma}
\label{wri wri}
 For C$^*$-algebras $A_0 \subseteq A_1 \subseteq A$, such that $A_0$ is relatively weakly injective
 in $A_1$, $A_1$ is relatively weakly injective in $A$,
then $A_0$ is relatively weakly injective in $A$.
\end{lemma}

The property of \emph{r.w.i.} is also closed under direct product. 
\begin{lemma}
 \label{prod wri}
If $(A_i)_{i\in I}$ is a net of $C^*$-algebras such that $A_i$ is relatively weakly injective in $B_i$ for all $i\in I$, then 
$\Pi_{i\in I}A_i$ is relatively weakly injective in $\Pi_{i\in I}B_i$.
\end{lemma}

In \cite{La2}, Lance establishes the following tensor product characterization of the WEP. The proof of the theorem is 
called \emph{The Trick}, and we will be using this throughout the paper. In the following, 
let $\mathbb{F}_{\infty}$ denote the free group 
with countably infinite generators, and $C^*\mathbb{F}_{\infty}$ be the full group C$^*$-algebra of $\mathbb{F}_{\infty}$. 

\begin{theorem}
\label{tensor wri}
A C$^*$-algebra $A$ has the WEP, if and only if 
\[
 A \underset{max}\otimes C^*\mathbb{F}_\infty = A \underset{min}\otimes C^*\mathbb{F}_\infty.
\]
\end{theorem}

As a consequence of the above theorem, we have the following result. 

\begin{corollary}
\label{wepwri}
A C$^*$-algebra $A$ has the WEP if and only if for any inclusion $A\subseteq B$, $A$ is relatively weakly injective in $B$. 
\end{corollary}

Similar to the WEP, the QWEP is also preserved by the relatively weak injectivity as following. 

\begin{lemma}
 If a C$^*$-algebra $A$ is relatively weakly injective in a QWEP C$^*$-algebra, then it is QWEP.
\end{lemma}

Although the WEP does not pass to the double dual, the QWEP is more flexible.

\begin{proposition}\label{dd}
 A C$^*$-algebra $A$ is QWEP if and only if $A^{**}$ is QWEP.
\end{proposition}

As a corollary of the above proposition, $\BB(\HH)^{**}$ is QWEP. Moreover we have the following equivalence.

\begin{corollary}\label{bl2qwep}
 A C$^*$-algebra $A$ is QWEP if and only if $A$ is relatively weakly injective in $\BB(\HH)^{**}$.
\end{corollary}

\subsection{Hilbert C$^*$-Modules}

The notion of Hilbert C$^*$-modules first appeared in a paper by Irving Kaplansky in 1953 \cite{Kap53}. 
The theory was then developed by the work of William Lindall Paschke in \cite{Pas73} . 
In this section we give a brief introduction to Hilbert C$^*$-modules and 
present some of their fundamental properties which we are going to use throughout this paper.

\begin{definition}
 Let $D$ be a C$^*$-algebra. An inner-product $D$-module is a linear space $E$ which is a right $D$-module with compatible scalar multiplication: $\lambda(xa)=(\lambda x)a= x(\lambda a)$, for $x\in E$, $a\in D$, $\lambda\in \mathbb{C}$, and a map
$(x,y)\longmapsto \langle x,y\rangle:E\times E\to D$ with the following properties:

\begin{enumerate}
 \item $\langle x, \alpha y+\beta z\rangle=\alpha\langle x,y\rangle+\beta \langle x,z\rangle\quad$ for $x$, $y$, $z\in E$ and $\alpha$, 
$\beta\in \mathbb{C}$;
 \item $\langle x,ya\rangle=\langle x,y \rangle a\quad$ for $x$, $y\in E$ and $a\in D$;
 \item $\langle y,x\rangle=\langle x,y\rangle^*\quad$ for $x$, $y\in E$;
 \item $\langle x,x \rangle \geq 0$; if $\langle x,x \rangle=0$, then $x=0$.
\end{enumerate}
\end{definition}

For $x\in E$, we let $\|x\|=\|\langle x,x\rangle\|^{1/2}$. It is easy to check that if $E$ is an inner-product $D$-module, then $\|\cdot\|$ is a norm on $E$.

\begin{definition}
 An inner-product $D$-module which is complete with respect to its norm is called a Hilbert $D$-module or a Hilbert C$^*$-module over the C$^*$-algebra $D$.
\end{definition}

Note that any C$^*$-algebra $D$ is a Hilbert $D$-module itself with the inner product $\langle x,y \rangle = x^*y$ for 
$x$ and $y$ in $D$. Another important example of a Hilbert C$^*$-module is the following.

\begin{example}
\normalfont
 Let $\mathcal{H}$ be a Hilbert space. Then the algebraic tensor product $\mathcal{H}\otimes_{alg}D$ 
can be equipped with a $D$-valued inner-product:
\[
\langle \xi\otimes a, \eta \otimes b\rangle=\langle\xi,\eta \rangle a^*b \qquad (\xi,\eta \in \mathcal{H}, a,b\in D).
\]
Let $\mathcal{H}_D=\mathcal{H}\otimes D$ be the completion of $\mathcal{H}\otimes_{alg}D$ with respect to the induced norm. 
Then $\mathcal{H}_D$ is a Hilbert $D$-module. 
\end{example}

Let $E$ and $F$ be Hilbert $D$-modules. Let $t$ be an adjointable map from $E$ to $F$, i.e. there exists 
a map $t^*$ from $F$ to $E$ such that
\[
\langle tx,y \rangle=\langle x,t^*y\rangle, \qquad \text{for }x\in E\text{ and } y\in F.
\]
One can easily see that $t$ must be right $D$-linear, that is, $t$ is linear and $t(xa)=t(x)a$ for all $x\in E$ and $a\in D$. 
It follows that any adjointable map is bounded, but the converse is not true -- a bounded $D$-linear map need not be adjointable.
Let $\mathcal{L}(E,F)$ be the set of all adjointable maps from $E$ to $F$, and we abbreviate $\mathcal{L}(E,E)$ to $\mathcal{L}(E)$. Note that $\mathcal{L}(E)$ 
is a C$^*$-algebra equipped with the operator norm.

Now we review the notion of compact operators on Hilbert $D$-modules, as an analogue to the compact 
operators on a Hilbert space. Let $E$ and $F$ be Hilbert $D$-modules. For every $x$ in $E$ and $y$ in $F$, define 
the map $\theta_{x,y}:E\to F$ by
\[
\theta_{x,y}(z)=y\langle x,z \rangle \qquad \text{for }z\in E.
\]
One can check that $\theta_{x,y}\in \mathcal{L}(E,F)$ and $\theta_{x,y}^*=\theta_{y,x}$. We denote by $\mathcal{K}(E,F)$ 
the closed linear subspace of $\mathcal{L}(E,F)$ spanned by $\{\theta_{x,y}: x\in E, y\in F\}$, and we abbreviate $\mathcal{K}(E,E)$ 
to $\mathcal{K}(E)$. We call the elements of $\mathcal{K}(E,F)$ \emph{compact} operators.

Let $E$ be a Hilbert $D$-module and $Z$ be a subset of $E$. We say that $Z$ is a \emph{generating set} for $E$ if the closed 
submodule of $E$ generated by $Z$ is the whole of $E$. If $E$ has a countable generating set, we say that $E$ is countably generated.

In \cite{Kas80}, Kasparov proves the following theorem known as the \emph{absorption theorem}, 
which shows the universality of $\HH_D$ in the category of Hilbert $D$-modules.
\begin{theorem}
 \label{absorption}
Let $D$ be a C$^*$-algebra and $E$ be a countably generated Hilbert $D$-module. Then $E\oplus \mathcal{H}_D \approx \mathcal{H}_D$, i.e. there exists an element $u\in\mathcal{L}(E\oplus \HH_D,\HH_D)$ such that $u^*u=1_{E\oplus \HH_D}$ and $uu^*=1_{\HH_D}$.
\end{theorem}

\begin{remark}
\label{universal}
\normalfont
Using the absorption theorem, for an arbitrary Hilbert $D$-module $E$, we have $\mathcal{L}(E\oplus \mathcal{H}_D) \simeq \mathcal{L(H}_D)$. 
Hence we have an embedding of
$\mathcal{L}(E)$ in $\mathcal{L(H}_D)$ and a conditional expectation from $\mathcal{L(H}_D)$ to $\mathcal{L}(E)$.
\end{remark}

Before we proceed to the main results of Hilbert C$^*$-modules, let us recall the notion of \emph{multiplier algebra} of a C$^*$-algebra.

\begin{definition}
 Let $A$ and $B$ be C$^*$-algebras. If $A$ is an ideal in $B$, we call $A$ an essential ideal if there is no nonzero ideal 
of $B$ that has zero intersection with $A$. Or equivalently if $b\in B$ and $bA=\{0\}$, then $b=0$.
\end{definition}

It can be shown that for any C$^*$-algebra $A$, there is a unique (up to isomorphism) maximal C$^*$-algebra which contains $A$ as an essential ideal. This algebra is called the multiplier 
algebra of $A$ and is denoted by $\mathcal{M}(A)$.

\begin{theorem}
\label{mult}
 If $E$ is a Hilbert $D$-module, then $\mathcal{L}(E)=\mathcal{M(K}(E))$.
\end{theorem}

Note that if $E=D$ for a unital C$^*$-algebra $D$, then $D=\mathcal{K}(D)$ and $\mathcal{L}(D)=\mathcal{M}(D)$.

In the special case where $E=\mathcal{H}_D$, we have 
\[
\mathcal{K(H}_D)\simeq \mathcal{K(H)}\underset{min}\otimes D = \K \otimes D,
\] 
where $\K = \mathcal{K(H)}$ is the C$^*$-algebra of the compact operators. Therefore, by Theorem ~\ref{mult} we have 
\[
\mathcal{L(H}_D)\simeq \mathcal{M}(\K\otimes D).
\]

In \cite{Kas80} Kasparov introduces a GNS type of construction in the context of Hilbert C$^*$-modules, known as the KSGNS construction (for Kasparov, Stinespring, Gelfand, Neimark, Segal) as follows.

\begin{theorem}
 \label{GNS}
 Let $A$ be a C$^*$-algebra, $E$ be a Hilbert $D$-module and let $\rho:A\to \mathcal{L}(E)$ be a completely positive map.
There exists a Hilbert $D$-module $E_{\rho}$, a $^*$-homomorphism $\pi_{\rho}:A\to \mathcal{L}(E_{\rho})$ and an element $v_{\rho}$ 
of $\mathcal{L}(E,E_{\rho})$, such that
\begin{align*}
\rho(a) =v^*_{\rho}\pi_{\rho}(a)v_{\rho} \qquad(a\in A), \\
\pi_{\rho}(A)v_{\rho}E\quad \text{is dense in}\quad E_{\rho}.
\end{align*}
\end{theorem}

As a consequence of the above theorem, Kasparov shows that given a C$^*$-algebra $D$, any separable $C^*$-algebra can be considered as a C$^*$-subalgebra of $\mathcal{L(H}_D)$. 
This indicates that $\mathcal{L(H}_D)$ plays the similar role in the category of Hilbert C$^*$-modules to that of $\BB(\HH)$ in the category of C$^*$-algebras.

\begin{proposition}
 Let $A$ be a separable C$^*$-algebra. Then there exists a faithful nondegenerate $^*$-homomorphism 
$\pi: A\to \mathcal{L(H}_D)$. 
\end{proposition}

As we see, $\mathcal{L}(\HH_D)$ plays the role of $\BB(\HH)$. Note that $\BB(\HH)$ is also a von Neumann algebra, 
but $\mathcal{L}(\HH_D)$ is not in general. Paschke in \cite{Pas73} introduces self-dual Hilbert C$^*$-modules to play the similar role in 
the von Neumann algebra context.

Let $E$ be a Hilbert $D$-module. Each $x\in E$ gives rise to a bounded $D$-module map $\hat {x}:E\to D$ defined by $\hat x(y)=\langle y,x \rangle$ for $y\in E$. 
We will call $E$ \emph{self-dual} if every bounded $D$-module map of $E$ into $D$ arises by taking $D$-valued inner products with some $x\in E$. 
For instance, if $D$ is unital, then it is a self-dual Hilbert $D$-module. Any self-dual Hilbert C$^*$-module is complete, but the converse is not true.

For von Neumann algebra $N$, it is natural to consider the the self-dual Hilbert $N$-module $E_N$, 
because of the following theorem from \cite{JS}.

\begin{theorem}
For a Hilbert C$^*$-module $E$ over a von Neumann algebra $N$, the following conditions are equivalent:
\begin{enumerate}
\item The unit ball of $E$ is strongly closed;
\item $E$ is principal, or equivalently, $E$ is an ultraweak direct sum of Hilbert C$^*$-modules $q_{\alpha}N$, for some projections $q_{\alpha}$;
\item $E$ is self-dual;
\item The unit ball of $E$ is weakly closed.
\end{enumerate}
\end{theorem}

We denote the algebra of adjointable maps on $E_N$ closed in the weak operator topology by $\mathcal{L}^w(E_N)$.

\begin{remark}
\label{e}
\normalfont
 According to \cite{Pas73} and the absorption theorem, for a von Neuamann algebra $N$, we have that $\mathcal{L}^w(E_N) = e \BB(\HH)\bar{\otimes} N e$ for some projection $e$. 
\end{remark}

\begin{remark}
\label{unitization}
\normalfont
 Let $N$ be a von Neumann subalgeba of $M$, such that $N=zM$ for some central projection $z\in M$. Then one can unitize the inclusion map $\iota: \BB(\ell_2)\bar{\otimes}N\hookrightarrow \BB(\ell_2)\bar{\otimes}M$. 
Indeed since $\BB(\ell_2)$ is a type I$_{\infty}$ factor, the projection $1\otimes z: \BB(\ell_2)\bar\otimes M \to  \BB(\ell_2)\bar\otimes N$ 
is properly infinite, and hence it is equivalent 
to identity on $\BB(\ell_2)\bar \otimes M$ \cite{Tak}. Let $1\otimes z = v^*v$, and $\id_{\BB(\ell_2)\bar \otimes M}=vv^*$. Note that $(1\otimes z)\circ \iota=\id_{\BB(\ell_2)\bar \otimes N}$.
Multiplying by $v$ from left and by $v^*$ from right, we get $v\iota v^*=\id_{\BB(\ell_2)\bar \otimes N}$.
\end{remark}

\subsection{Kirchberg's observations on the multiplier algebra}

In this section, we explore Kirchberg's seminal paper on non-semisplit extensions in detail. In particular we show the factorization property explicitly for readers' convenience. 

Let $A$, $B$ and $C$ be C$^*$-algebras. We say a map $h:A\to B$ \emph{factors through $C$ approximately via ucp maps in point-norm topology} if there exist ucp maps $\varphi_n:A\to C$ and $\psi_n: C\to B$ such that the following diagram commutes approximately 
in point-norm topology.
\[
\xymatrix{
   A\ar [dr]_{\varphi_n} \ar [rr]^{h} && B\\
    &C\ar [ur]_{\psi_n}&\\
  }   
\]
i.e.  $\|(\psi_n\circ\varphi_n)(x)-h(x)\|\to 0$ for all $x\in A$.

\begin{theorem}
\label{factor1}
Let $A$ be a C$^*$-algebra and $\mathcal{M}(A)$ be its multiplier algebra. Then the identity map on $\mathcal{M}(A)$
 factors through $\ell_{\infty}(A)$ approximately via ucp maps in point-norm topology.
\end{theorem}

\begin{proof}[Sketch of proof] 
Let us fix an approximate unit $\{h_i\}$ for $A$. Assume $B$ is a separable unital subalgebra of $\mathcal{M}(A)/A$. Let
$\pi: \M(A) \to \M(A)/A$ be the quotient map, and $D$ be a separable unital subalgebra of $\mathcal{M}(A)$ 
such that $\pi(D)=B$ and $D$ contains the approximate unit. 
Using these objects, Kirchberg constructed unital completely positive contractions 
$V_n:\ell_{\infty}(A) \to \mathcal{M}(A)$, and linear maps 
$W_n:\mathcal{M}(A)\odot C_0(0,1]\to\ell_{\infty}(A)\otimes C_0(0,1]$ as follows:

Let $(u_i)$ be a dense sequence in the unitary group of $D$. One can find finite convex combinations $\{g_i\}$ of elements $\{h_i\}$ such that 
\begin{align}
g_ng_{n+1}&=g_n, \\
g_nh_n&=h_n, \\
\|g_nu_j-u_jg_n\|&\leq 2^{-2n}, \quad \text{if} \quad j\leq n. 
\end{align}

Now we define $d_0^{(n)}:=g_n^{1/2}$ and $d_k^{(n)}:=(g_{k+n}-g_{k+n-1})$ for $k\geq 1$, and put $V_n(b_1,b_2,...)=\sum_{k=0}^{\infty}d_k^{(n)}b_kd_k^{(n)}$, where the sum is taken in $A^{**}$. Define
\[
X_n:=(x_1^{(n)},x_2^{(n)},...)
\]
 where
 \[
 x_1^{(n)}=g_{n+2}, \quad x_k^{(n)}=g_{k+n+2}-g_{k+n+3}
 \]
Let $\Delta: \M(A) \to \ell_\infty(\M(A))$ be defined by $\Delta(a)=(a,a,...)$ for $a\in \mathcal{M}(A)$. Let $f_0\in C_0(0,1]$ be such that $f_0(x)=x$ for $x\in (0,1]$. Now define $W_n(a\otimes f)=f(\lambda_n)(\Delta(a)\otimes id)$, where $\lambda_n:=X_n\otimes f_0\in A\otimes C_0(0,1]$ for $a\in \mathcal{M}(A)$ and $f\in C_0(0,1]$.
It is easy to check that the maps $V_n$ and $W_n$ have the following properties:
\begin{enumerate}
\item $(V_n\otimes \id) (W_n(b))=(V_n\otimes \id)(\Delta \otimes \id )(b)$ for $b\in \mathcal{M}(A)\odot C_0(0,1]$.  (Use (2.1) and (2.2).)

\item The map 
\[
a\in \mathcal{M}(A)\longmapsto W_n(a\otimes f)+(c_0(A)\otimes C_0(0,1])
\]
lifts to a completely contractive map
\[
U_{n,f}:\mathcal{M}(A)\to \ell_{\infty}(A)  \otimes C_0(0,1] .
\]

\item $\lim_{n\to \infty} V_n(a)=a$, for $a\in {\pi}^{-1}(B)$. (Use (2.3).)
\end{enumerate}

Moreover for $f\in C_0(0,1]$ such that $0\leq f\leq 1$ and $f(1) =1$, let $g=\sqrt f$. One can define $U_{n,f}(a)=g(\lambda _n)(\Delta (a)\otimes \id)g(\lambda _n)$. Therefore, we have the following commutative diagram:
\[
\xymatrix{
   \mathcal{M}(A)\ar [drr]_{U_{n,f}} \ar [r]^{\id\otimes 1\qquad} & \mathcal{M}(A)\odot C_0(0,1] \ar [dr]^{W_n}\ar [rr]^{\id}&& \mathcal{M}(A)\odot C_0(0,1] \ar [r] ^{\qquad{\id\otimes\ev}} &\mathcal{M}(A)\\
    &&\ell_{\infty}(A)  \otimes C_0(0,1] \ar [ur]^{V_n\otimes \id}&&\\
  }   
\]
where $\ev$ is the evaluation of $f(x)$ in $C_0(0,1]$ at 1. From the properties above, we have
\begin{align*}
\lim_{n\to \infty}(\id\otimes\ev)(V_n\otimes \id)(W_n(b))&=\lim_{n\to \infty}(\id\otimes\ev)(V_n\otimes \id(\Delta(a)\otimes \id(f)))\\
&=\lim_{n\to \infty}V_n(a)f(1)=a.
\end{align*}
Since $\ell_{\infty}(A)  \otimes C_0(0,1]$ factors though $\ell_{\infty}(A)$, this proves that $\mathcal{M}(A)$ factors through $\ell_{\infty}(A)$ approximately via ucp maps.
\end{proof}
Using the above theorem, we can establish the following result on the relation between $\M(A)$ and $A^{**}$. 
\begin{corollary}\label{ma}
 Suppose $A$ is a C$^*$-algebra and $\mathcal{M}(A)$ is its multiplier algebra. Then $\mathcal{M}(A)$ is relatively weakly injective in $A^{**}$.
\end{corollary}

\begin{proof}
 With the notation of the previous theorem, since $U_{n,f}$ lifts $W_n$, we have
\[\lim_{n\to \infty}(\id\otimes\ev)(V_n\otimes \id)(U_{n,f}(a))= a,\]
for $a\in \pi^{-1}(B)$. 

Since there is a natural inclusion $\mathcal{M}(A)\subset A^{**}$, we can define $\tilde {U}_{n,f}:A^{**}\to \ell_{\infty}(A^{**})  \otimes C_0(0,1]$ as an extension of ${U_{n,f}}$, by
 $\tilde{U}_{n,f}(a)=g(\lambda _n)(\Delta (a)\otimes \id)g(\lambda _n)$ for all $a\in A^{**}$. Using the fact that $\ell_{\infty}(A)  \otimes C_0(0,1]$ is \emph{r.w.i.} in $\ell_{\infty}(A^{**})  \otimes C_0(0,1]$, we get the following diagram 
\[
\xymatrix{
   \mathcal{M}(A)\ar@ {^{(}->} [dd] \ar [drr]_{U_{n,f}} \ar [r] & \mathcal{M}(A)\odot C_0(0,1] \ar [dr]^{W_n}\ar [rr]^{\id}&& \mathcal{M}(A)\odot C_0(0,1] \ar [r] ^{\qquad{\id\otimes\ev}} &\mathcal{M}(A)\\
    &&\ell_{\infty}(A)  \otimes C_0(0,1] \ar[ur]^{V_n\otimes \id}&&\\
    A^{**} \ar [rr]_{\tilde{U}_{n,f}}&& \ell_{\infty}(A^{**})  \otimes C_0(0,1] \ar [u] &
  }   
\]
which commutes locally: let $\e$ be arbitrary, $F$ a finite-dimensional subspace of $A^{**}$, and $F_0=F\cap \M(A)$. Then we get a net $\Lambda=(F_0, F, \e)$.  Define $\psi_{n,\lambda}: A^{**}\to \mathcal{M}(A)$ locally, as the composition of the following maps 
\[
  F\overset{\tilde{U}_{n,f}}\hookrightarrow A^{**}\to \ell_\infty(A^{**})\otimes C_0(0,1]\to \ell_\infty(A)\otimes C_0(0,1]\overset{V_n\otimes id}\longrightarrow\mathcal{M}(A)\odot C_0(0,1]\overset{\id\otimes\ev}\longrightarrow\mathcal{M}(A).
\]
Then we have
\[
\lim_{n, \lambda}\psi_{n,\lambda}(1)=1.
\]
Let $\psi:=\lim_{n, \lambda}\psi_{n,\lambda}: A^{**}\to \mathcal{M}(A)^{**}$ in the weak $^*$-topology. Then $\psi$ gives the required conditional expectation, and this proves that $\mathcal{M}(A)$ is \emph{r.w.i.} in $A^{**}$.
\end{proof}

\section{Module version of the weak expectation property}

The notion of \emph{r.w.i.} is a paired relation between a C$^*$-subalgebra and its parent C$^*$-algebra. If the 
parent C$^*$-algebra is $\BB(\HH)$, the \emph{r.w.i.} property is equivalent to the WEP. By carefully choosing a parent C$^*$-algebra, 
we can define the notion of WEP relative to a C$^*$-algebra. 
 
Let $\mathcal{C}$ be a collection of inclusions of unital C$^*$-algebras $\{(A\subseteq X)\}$.

For a C$^*$-algebra $D$, there are two classes of objects that we will discuss throughout this paper. 
\begin{enumerate}
 \item $\mathcal{C}_1 = \{ A\subseteq \mathcal{L}(E_D)  \}$, where $E_D$ is a Hilbert $D$-module.
 \item $\mathcal{C}_2 = \{ A\subseteq \mathcal{L}^w(E{_{D^{**}})}\}$, where $E_{D^{**}}$ is a self dual Hilbert $D^{**}$-module.
 \end{enumerate}

\begin{definition}
 A C$^*$-algebra $A$ is said to have the $D$WEP$_i$ for $i=1,2$, if there exists a pair of inclusions $A\subseteq X $ in $\mathcal{C}_i$ 
such that $A$ is relatively weakly injective in $X$.
\end{definition}

Notice that the notion of $D$WEP is a \emph{r.w.i.} property. By Corollary ~\ref{wepwri}, the WEP implies the $D$WEP$_i$, for $i=1,2$. 
Also, inherited from \emph{r.w.i.} property, we have the following lemmas for $D$WEP. 

\begin{lemma}
 \label{wri wep}
Let $A_0$ and $A_1$ be C$^*$-algebras such that $A_0$ is relatively weakly injective in $A_1$. If $A_1$ has the $D$WEP$_i$ for $i=1,2$, then so does $A_0$.
\end{lemma}

\begin{proof}
 Since $A_1$ has the $D$WEP$_i$, there exists a pair of inclusions $A_1\subseteq X$ in $\mathcal{C}_i$ 
such that $A_1$ is \emph{r.w.i.} in $X$, for $i=1,2$. 
By Lemma ~\ref{wri wri}, $A_0$ is \emph{r.w.i.} in $X$. Therefore the result follows.
\end{proof}

\begin{remark}
\normalfont
 By the absorption theorem and Remark ~\ref{universal} and ~\ref{e}, $\mathcal{L}(E_D)$ is 
\emph{r.w.i.} in some $\mathcal{L}(\HH_D)$ and $\mathcal{L}^w(E{_{D^{**}})}$ is \emph{r.w.i.} in some 
$\BB(\HH)\bar\otimes D^{**}$. Sometimes it is more convenient to consider the $D$WEP$_1$ as the relatively weak injectivity 
in $\mathcal{L}(\HH_D)$, and the $D$WEP$_2$ as the relatively weak injectivity in $\BB(\HH)\bar\otimes D^{**}$, because of the 
concrete structures. 
\end{remark}

\begin{example}\label{kd}
\normalfont
 From the above, all WEP algebras have $D$WEP$_i$ for arbitrary C$^*$-algebra $D$. Also, $D$ has the $D$WEP$_i$ trivially
for $1$-dimensional Hilbert space $\HH$. Our first nontrivial example of $D$WEP$_i$ is $\mathcal{K} \otimes D$. For the first class $\C_1$,
$\mathcal{K} \otimes D$ is a principle ideal of $\mathcal{L}(\mathcal{H}_D)$, and thus is \emph{r.w.i.} in $\mathcal{L}(\mathcal{H}_D)$.
For the second class $\C_2$, note that $(\mathcal{K} \otimes D)^{**} = \BB(\mathcal{H}) \bar{\otimes} D^{**}$, so $\mathcal{K} \otimes D$ is \emph{r.w.i.} 
in $\BB(\mathcal{H}) \bar{\otimes} D^{**}$. By universality of $\mathcal{L}(\mathcal{H}_D)$
and $\BB(\mathcal{H}) \bar{\otimes} D^{**}$, $\mathcal{K} \otimes D$ has the $D$WEP$_i$ for both $i=1,2$. 
% and $B(\mathcal{H}) \bar{\otimes} D^{**}$ is weakly cp complemented in $\mathcal{L}(\mathcal{H}_D^{**})$. 
\end{example}

Because of the injectivity of $\BB(\HH)$, we see that the notion of WEP does not depend on the representation $A\subseteq \BB(\mathcal{H})$. By constructing a universal object
in the classes $\C_i$, we can define the $D$WEP$_i$ independent of inclusions.

\begin{lemma}
A C$^*$-algebra $A$ has the $D$WEP$_i$ for some inclusion $A\subseteq X $ in $\mathcal{C}_i$, if and only if there exists a universal object
$X^u$ and $A\subseteq X^u$ in $\C_i$, such that 
\begin{enumerate}
 \item $A$ is relatively weakly injective in $X^u$; 
 \item If $A$ is relatively weakly injective in some $X$, then there exists a ucp map from $X^u$ to $X$, 
which is identity on $A$. 
\end{enumerate}

\end{lemma}

\begin{proof}
 The ``if'' part is trivial. For the ``only if'' part, take $\C_1$ for example. The proof of the other case is similar. For all ucp maps $\rho:  A\to \mathcal{L}(E_D)$, by KSGNS construction there exists a Hilbert $D$-module $E_\rho$ and a $^*$-homomorphism $\pi_\rho: A \to \mathcal{L}(E_\rho)$. 
Let $E^u_D = \bigoplus_{\rho} E_\rho$. Then any $\mathcal{L}(E_D)$ containing $A$ can be embedded into $\mathcal{L}(E^u_D)$, and
there exists a truncation $\mathcal{L}(E^u_D) \to \mathcal{L}(E_D)$. Now suppose $A$ is \emph{r.w.i.} in some $\mathcal{L}(E_D)$. Then it is also \emph{r.w.i.}
 in $\mathcal{L}(E^u_D)$. Hence we complete the proof. 
\end{proof}

Following Lance's tensor product characterization Theorem ~\ref{tensor wri}, we have a similar result for the $D$WEP$_i$, for $i=1,2$. We only 
present the result for the first class. The other case can be proved similarly.

Let $A\subseteq\mathcal{L}(E^u_D)$ be the universal representation. We define a tensor norm $max^D_1$ on $A\otimes C^*\F_{\infty}$ 
to be the norm induced from the inclusion 
$A\otimes C^*\F_{\infty} \subseteq \mathcal{L}(E^u_D)\otimes_{max}C^*\F_{\infty}$ isometrically. This induced norm is categorical 
in the sense that if $\phi$ is a ucp map from $A$ to $B$, then $\phi\otimes \id$ extends a ucp map from 
$A\otimes_{max_1^D} C^*\F_{\infty}$ to $B\otimes_{max_1^D} C^*\F_{\infty}$. Indeed, let $\iota$ be the inclusion map
from $B$ to its universal representation $\mathcal{L}^B(E^u_D)$, then $\iota \circ \phi$ is a ucp map from 
$A$ to $\mathcal{L}^B(E^u_D)$. By KSGNS and the construction of $\mathcal{L}^A(E^u_D)$, there exists a ucp map from 
$\mathcal{L}^A(E^u_D)$ to $\mathcal{L}^B(E^u_D)$ extending the map $\iota \circ \phi$. Hence we have a composition of 
ucp maps
\[
 A\underset{max^D_1}\otimes C^*\F_{\infty} \subseteq \mathcal{L}^A(E^u_D)\underset{max}\otimes C^*\F_{\infty} \to \mathcal{L}^B(E^u_D)\underset{max}\otimes C^*\F_{\infty}, 
\]
whose image is $B\otimes_{max_1^D} C^*\F_{\infty}$.

\begin{theorem}
\label{tensor wep}
 A C$^*$-algebra $A$ has the the $D$WEP$_1$, if and only if
\[
 A\underset{max^D_1}\otimes C^*\F_{\infty} = A\underset{max} \otimes C^*\F_{\infty}.
\]
\end{theorem}

\begin{proof}
First, suppose $A$ has the $D$WEP$_1$, then $A$ is \emph{r.w.i.} in $\mathcal{L}(E^u_D)$. That is, there exists a ucp map $\varphi : \mathcal{L}(E^u_D) \to A^{**}$ such that $\varphi |_A = \id_A$. 
Then $\varphi \otimes \id$ gives a ucp map from $\mathcal{L}(E^u_D) \otimes_{max} C^*\mathbb{F}_\infty$ to $A^{**}\otimes_{max} C^*\mathbb{F}_\infty$. Therefore the map $\varphi |_A \otimes \id$, defined 
on the algebraic tensor product, extends to a ucp map from
$A\otimes_{max_1^D} C^{*}\mathbb{F}_{\infty}$ to $A^{**}\otimes_{max} C^{*}\mathbb{F}_{\infty}$, whose image is $A\otimes_{max} C^{*}\mathbb{F}_{\infty}$. 

To prove the other direction, suppose $A\otimes_{max_1^D} C^{*}\mathbb{F}_{\infty}=A\otimes_{max} C^{*}\mathbb{F}_{\infty}$, and let $A\subseteq \BB(\mathcal{H})$be the universal representation, i.e. $A''=A^{**}$.
Let $\pi$ be a representation of $C^*\F_\infty$ to $\BB(\mathcal{H})$ with $\pi(C^*\mathbb{F}_\infty)'' = A'$. 
Then there exists a ucp map $A\otimes_{max_1^D} C^{*}\mathbb{F}_{\infty}=A\otimes_{max} C^{*}\mathbb{F}_{\infty} \to  \BB(\mathcal{H})$.
Since $A\otimes _{max_1^D} C^{*}\mathbb{F}_{\infty}\subseteq \mathcal{L}(E^u_D)\otimes_{max} C^{*}\mathbb{F}_{\infty}$, by Arveson's extension theorem, we can extend the above ucp map to 
$\Psi: \mathcal{L}(E^u_D)\otimes_{max} C^{*}\mathbb{F}_{\infty}\to \BB(\mathcal{H})$. Now define a map $\psi$ on $\mathcal{L}(E^u_D)$ by
$\psi (T):=\Psi (T \otimes 1)$, for $T\in \mathcal{L}(E^u_D)$. Then $\psi$ is a ucp extension of $\id_A$. 
Since $\Psi$ is a $C^*\mathbb{F}_\infty$-bimodule map, we have $\psi(T) \pi (x) = \Psi(T\otimes x) = \pi(x)\psi(T)$ for $T\in \mathcal{L}(E^u_D)$ and $x\in C^*\mathbb{F}_\infty$,
i.e., $\psi(T) \in \pi(C^*\mathbb{F}_\infty)' = A^{**}$. This completes the proof.
\end{proof}

It is natural to explore the relationship between $D$WEP$_1$ and $D$WEP$_2$. We have the following.

\begin{theorem}\label{12}
If a C$^*$-algebra $A$ has the $D$WEP$_1$, then it also has the $D$WEP$_2$.
\end{theorem}

In fact, the converse of the above theorem is not true, and we will give a counterexample in Section 5. 

To prove the above theorem, we need following lemmas.

\begin{lemma}
 \label{ucp}
Suppose that the identity map on a C$^*$-algebra $A$ factors through a C$^*$-algebra $B$ approximately via ucp maps in point-norm topology, i.e. there exist 
two nets of ucp maps $\phi_i: A \to B $ and $\psi_i: B \to A$, such that $\|\psi_i \circ \phi_i (x) - x \| \to 0$ for $x\in A$. 
If $B$ has the $D$WEP$_i$, then so does 
$A$. 
\end{lemma}

\begin{proof} Following Kirchberg's method, it suffices to show 
$A\otimes_{max_i^D} C^{*}\mathbb{F}_{\infty} = A\otimes_{max} C^{*}\mathbb{F}_{\infty}$. Since we have ucp maps 
$\phi_i: A \to B $ and $\psi_i: B \to A$, such that $\|\psi_i \circ \phi_i (x) - x \| \to 0$ for $x\in A$, we have ucp maps
$\phi_i \otimes \id : A\otimes_{max_i^D} C^{*}\mathbb{F}_{\infty} \to B \otimes_{max_i^D} C^{*}\mathbb{F}_{\infty}$ and ucp maps
$\psi_i \otimes \id : B\otimes_{max} C^{*}\mathbb{F}_{\infty} \to A \otimes_{max} C^{*}\mathbb{F}_{\infty}$. 
Since $B$ has the $D$WEP$_i$, by Theorem ~\ref{tensor wep}, 
we have $B\otimes_{max_i^D} C^{*}\mathbb{F}_{\infty} = B\otimes_{max} C^{*}\mathbb{F}_{\infty}$. Therefore we have ucp maps 
$A\otimes_{max_i^D} C^{*}\mathbb{F}_{\infty} \to A\otimes_{max} C^{*}\mathbb{F}_{\infty}$ defined by the composition of the maps according to the following diagram
\[
A\underset{max^D_i}\otimes C^{*}\mathbb{F}_{\infty}\overset{\phi_i\otimes \id}\longrightarrow B \underset{max^D_i}\otimes C^{*}\mathbb{F}_{\infty}=B \underset{max}\otimes C^{*}\mathbb{F}_{\infty}\overset{\psi_i\otimes \id}\longrightarrow A\underset{max}\otimes C^{*}\mathbb{F}_{\infty}.
\]
This net of maps converges to the identity. Hence we get the result.  
\end{proof}

Another lemma we need is that the $D$WEP$_i$ property is preserved under the direct product.

\begin{lemma}
 \label{prod wep}
If $(A_i)_{i \in I}$ is a net of $C^*$-algebras with the $D$WEP$_i$, then $\prod _{i\in I} A_i$ has the $D$WEP$_i$.
\end{lemma}

\begin{proof}
We will prove the result for the $D$WEP$_1$. The proof of the other case is similar. 
 Since each $A_i$ has the $D$WEP$_1$, there exists an inclusion $A_i \subseteq \mathcal{L}((E_i)_D) $ such that $A_i$ is \emph{r.w.i.} in $\mathcal{L}((E_i)_D)$. 
By Lemma ~\ref{prod wri}, $\prod _{i\in I} A_i$ is \emph{r.w.i.} in $\prod_{i\in I} \mathcal{L}((E_i)_D)$. 
Since $\mathcal{L}(\oplus_{i \in I} (E_i)_D)$ contains $\prod_{i\in I} \mathcal{L}((E_i)_D)$ and it has a conditional expectation onto $\prod \mathcal{L}((E_i)_D)$, 
$\prod _{i\in I} A_i$ is also \emph{r.w.i.} in $\mathcal{L}(\oplus_{i\in I} (E_i)_D)$. 
Therefore $\prod _{i\in I} A_i$ has the $D$WEP$_i$.
\end{proof}

Kirchberg\cite{Ki1} shows that for a C$^*$-algebra $A$, the multiplier algebra $\M(A)$ factors through $\ell_{\infty}(A)$ 
approximately by ucp maps (Theorem ~\ref{factor1}). Using this fact, we have the following. 

\begin{corollary}\label{md}
 Suppose that the C$^*$-algebra $A$ has the $D$WEP$_i$, for $i = 1, 2$. Then the multiplier algebra $\M(A)$ also has the $D$WEP$_i$, for $i = 1, 2$.
\end{corollary}

\begin{proof}
 By Theorem ~\ref{factor1}, $\M(A)$ factors through $\ell_\infty(A)$ approximately via ucp maps in point-norm topology. 
Since $A$ has the $D$WEP$_i$, $\ell_\infty(A)$ has $D$WEP$_i$ by Lemma ~\ref{prod wep}. Therefore 
by Lemma ~\ref{ucp}, $\M(A)$ also has the $D$WEP$_i$.   
\end{proof}

Now we are ready to see the proof of the theorem.

\begin{proof}
[Proof of Theorem ~\ref{12}]
It suffices to show that $\mathcal{L}(\mathcal{H}_D)$ has the $D$WEP$_2$. Notice that $\mathcal{L}(\mathcal{H}_D) = \M(\mathcal{K}\otimes D)$, and also $\M(\mathcal{K}\otimes D)$ factors through 
$\ell_\infty(\mathcal{K}\otimes D)$ approximately via ucp maps in point-norm topology. Since 
$\mathcal{K}\otimes D$ has the $D$WEP$_2$ by Remark ~\ref{kd}, and hence $\ell_\infty(\mathcal{K}\otimes D)$ by Lemma ~\ref{prod wep}. 
By Corollary ~\ref{md}, $\M(\mathcal{K}\otimes D)$ has the $D$WEP$_2$. 
\end{proof}

\begin{remark}
\label{vn wep}
\normalfont
Note that $D^{**}$WEP$_1$ implies $D$WEP$_2$. Indeed having $D^{**}$WEP$_1$ is equivalent to being \emph{r.w.i.} in $\LL(\HH_{D^{**}})=\M(\K\otimes D^{**})$, and having 
$D$WEP$_2$ is equivalent to being \emph{r.w.i.} in $\BB(\HH)\bar\otimes D^{**}$. Note that $\K\otimes D^{**}$ is \emph{r.w.i.} in $\BB(\HH)\bar\otimes D^{**}$. 
By Corollary ~\ref{md}, we have $\M(\K\otimes D^{**})$ has the $D$WEP$_2$ as well. 
\end{remark}

Now we investigate some properties of module WEP. The first result is that the module WEP is stable under tensoring 
with a nuclear C$^*$-algebra, similar to the classical case.

\begin{proposition}
\label{min wep}
 For a C$^*$-algebra $D$, the following properties hold:
\begin{enumerate}
 \item If a $C^*$-algebra $A$ has the $D$WEP$_1$, and $B$ is a nuclear C$^*$-algebra, then $A\otimes_{min}B$ has the $D$WEP$_1$ as well.
 \item If von Neumann algebras $M$ and $N$ have the $C$WEP$_2$ and $D$WEP$_2$ respectively, then $M \bar{\otimes} N$ 
has the $(C\otimes_{min}D)$WEP$_2$. 
 \end{enumerate}
\end{proposition}

\begin{proof}
  (1) Since $A$ has the $D$WEP$_1$ and $B$ is a nuclear, we have $A$ is \emph{r.w.i.} in $\mathcal{L}(\mathcal{H}_D)$, 
and $B$ is \emph{r.w.i.} in $\BB(\HH)$. Therefore we have ucp maps $A\otimes_{min}B \to \mathcal{L}(\HH_D) \otimes_{min} \BB(\HH) =\mathcal{L}(\tilde{\HH}_D) \to A^{**}\otimes_{min}B^{**}$. 
Note that $B$ is nuclear and hence exact, so the inclusion map $A^{**} \otimes B^{**} \hookrightarrow (A\otimes_{min}B)^{**}$ is min-continuous. Therefore 
$A\otimes_{min} B$ is \emph{r.w.i.} in $\mathcal{L}(\tilde{\HH}_D)$. 

  (2) Since $M$ is \emph{r.w.i.} in $\mathcal{L}^w(\HH_{C^{**}})$ and $N$ is \emph{r.w.i.} in $\mathcal{L}^w(\HH_{D^{**}})$, 
we have ucp maps $M\bar{\otimes}N \to \mathcal{L}^w(\HH_{C^{**}}) \bar{\otimes} \mathcal{L}^w(\HH_{D^{**}}) = \mathcal{L}^w(\HH_{C^{**}\bar{\otimes}D^{**}})\to M\bar{\otimes}N \to (M\bar{\otimes}N)^{**}$. 
Note that $C\otimes_{min} D$ is weak $^*$-dense in $C^{**}\bar{\otimes}D^{**}$. Therefore we have a normal conditional expectation 
$(C\otimes_{min}D)^{**} \to C^{**}\bar{\otimes}D^{**}$, and hence $C^{**}\bar{\otimes}D^{**}$ is \emph{r.w.i.} in $(C\otimes_{min}D)^{**}$. 
Therefore $\mathcal{L}^w(\HH_{C^{**}\bar{\otimes}D^{**}})$ is \emph{r.w.i.} in $\mathcal{L}^w(\HH_{(C\otimes_{min}D)^{**}})$, and hence 
$M\bar{\otimes}N$ is \emph{r.w.i.} in $\mathcal{L}^w(\HH_{(C\otimes_{min}D)^{**}})$.
\end{proof}

As a consequence of Corollary ~\ref{md}, we have the transitivity property of $D$WEP. 

\begin{proposition}\label{tran}
 If $A$ has the $B$WEP$_i$, and $B$ has the $C$WEP$_i$, then $A$ has the $C$WEP$_i$, for $i=1,2$. 
\end{proposition}

\begin{proof}
Since $B$ has the $C$WEP$_i$, then so does $\mathcal{K}\otimes B$, and hence so does $\M(\mathcal{K}\otimes B)$ by Corollary ~\ref{md}. 
Since $A$ has the $B$WEP$_i$, $A$ is \emph{r.w.i.} in some $\mathcal{L}(\mathcal{H}_B) = \M(\mathcal{K}\otimes B)$. 
By the transitivity of \emph{r.w.i.}, we conclude that $A$ has the $C$WEP$_i$, for $i=1,2$. 
\end{proof}

\begin{corollary}\label{trand}
 If $A$ has the $D$WEP$_1$, and $D$ has the WEP, then $A$ has the WEP. 
\end{corollary}

\begin{proof}
 It suffices to show that $\LL(\HH_D)=\M(\K\otimes D)$ has the WEP. This is obvious since if $D$ has the WEP, then so does $\K\otimes D$ and hence $\M(\K\otimes D)$ has the WEP.
\end{proof}

\begin{remark}
 \normalfont
The previous result is not necessarily true for the WEP$_2$ case, since $\BB(\ell_2)\bar \otimes D^{**}$ may not have the WEP, for instance for $D=\BB(\ell_2)$. See Example ~\ref{Bl2} for the proof.
\end{remark}

In \cite{Ju96}, Junge shows the following finite dimensional characterization of the WEP. 

\begin{theorem}
 The C$^*$-algebra $A$ has the WEP if and only if for arbitrary finite dimensional subspaces $F\subset A$ and $G\subset A^*$, 
and $\e > 0$, there exist matrix algebra $M_m$ and ucp maps $u: F\to M_m$, $v: M_m \to A/G^{\perp}$, such that 
\[
 \|v\circ u-q_G \circ \iota_F\|<\e,
\]
where $\iota_F: F\to A$ is the inclusion map and $q_G: A \to A/G^{\perp}$ is the quotient map.
\end{theorem}

We have a similar result for the module WEP as follows.

\begin{theorem}
The C$^*$-algebra $A$ has the $D$WEP$_1$ if and only if for arbitrary finite dimensional subspaces $F\subset A$ and $G\subset A^*$, 
and $\e > 0$, there exist matrix algebra $M_m(D)$ and ucp maps $u: A\to M_m(D)$, $v:M_m(D) \to A/G^{\perp}$, such that 
\[
 \|v\circ u|_F-q_G \circ \iota_F\|<\e,
\]
where $\iota_F: F\to A$ is the inclusion map and $q_G: A \to A/G^{\perp}$ is the quotient map.
\end{theorem}

For the $D$WEP$_2$ case, we will replace the matrix algebra $M_m(D)$ by $M_m(D^{**})$.

\begin{proof}
 $\Leftarrow$: From the assumption, we get a net of maps $u$ and $v$ over $(F, G, \e)$. Taking the direct product of all such $u$, and 
one w$^*$-limit of $v$, we have ucp maps 
$A \to \Pi_{(F,G,\e)} M_m(D) \to A^{**}$, whose composition is identity on $A$, and hence $A$ is \emph{r.w.i.} in $\Pi_{(F,G,\e)} M_m(D)$. By Lemma ~\ref{prod wep}, $\Pi_{(F,G,\e)} M_m(D)$ has the $D$WEP$_1$ since 
$M_m(D)$ does. Therefore $A$ has the $D$WEP$_1$. 

 $\Rightarrow$: $A$ has the $D$WEP$_1$, and hence we have $A \to \mathcal{L}(\HH_D) \to A^{**}$. Let $\sigma_I$ be the composition of the inclusion maps
\[ 
A\hookrightarrow \mathcal{L}(\HH_{D})\hookrightarrow\Pi_I (M_{m(i)}(D))\hookrightarrow\Pi_{I}(M_{m(i)}(D^{**})).
\]
Note that each of the inclusions above is \emph{r.w.i.}. By taking the duals, we have
\[
 \varphi_I: A^* \overset{r.w.i.}\hookrightarrow \mathcal{L}(\HH_{D})^* \overset{r.w.i.}\hookrightarrow \Pi_I (M_{m(i)}(D))^* \overset{r.w.i.}\hookrightarrow \Pi_I (M_{m(i)}(D^{**}))^*  = \ell_1^I(\mathcal{S}_1^{m(i)}(D^*))^{**}.
\]
By the local reflexivity principle, for arbitrary $F$, $G$ and $\e$ as in the theorem, there exists a map $\alpha^{\e}_I : G \to \ell_1(I, \mathcal{S}_1(D^*))$, such that
\[
 |\langle \alpha^{\e}_I(g), \sigma_I(f)\rangle - \langle \varphi_I(g), \sigma_I(f)\rangle| < \e \|f\|\|g\|,
\]
for $f\in F$ and $g\in G$. By carefully choosing an Auerbach basis for the finite dimensional spaces, 
we can have the above relation on a finite subset $I_0 \subset I$, i.e.
\[
 |\langle \alpha^{\e}_{I_0}(g), \sigma_{I_0}(f)\rangle - \langle \varphi_{I_0}(g), \sigma_{I_0}(f)\rangle| < \e \|f\|\|g\|.
\]
By the \emph{r.w.i.} property of $\sigma_I$ we have $\langle \varphi_I(g), \sigma_I(f)\rangle = \langle g, f\rangle$. 
Therefore for $f$ and $g$ with norm $1$, we have $|\langle \alpha^{\e}_{I_0}(g), \sigma_{I_0}(f)\rangle-\langle g, f\rangle|<\e$, and hence $|\langle g, {\alpha^{\e*}_{I_0}} \circ \sigma_{I_0}(f) - f\rangle |< \e$. Let $u = \sigma_{I_0}$ and $v = {\alpha^{\e*}_{I_0}}$. Then we have the desired result.   
\end{proof}

\section{Module version of QWEP}

\begin{definition}
A $C^*$-algebra $B$ is said to be $D$QWEP$_i$ if it is the quotient of a $C^*$-algebra $A$ with $D$WEP$_i$ for $i=1,2$. 
\end{definition}

Similar to the $D$WEP$_i$, we have a tensor characterization for $D$QWEP$_i$ for $i=1,2$ as follows. First we need the following result due to Kirchberg.

\begin{lemma}[\cite{Ki1} Corollary 3.2 (v)]
\label{ki md}
If $\phi:A\to B^{**}$ is a ucp map such that $\phi$ maps the multiplicative domain $\md(\phi)$ of $\phi$ onto a C$^*$-subalgerba $C$ of $B^{**}$ containing $B$ as a subalgebra, then the C$^*$-algebra $\md(\phi)\cap \phi^{-1}(B)$ is relatively weakly injective in $A$.
\end{lemma}

We only prove the tensor characterization for $D$QWEP$_1$. The proof of the other case is similar.

\begin{theorem}
\label{tensor qwep}
 Let $\CF\subset \LL(\HH^u_D)$ be the universal representation. The following statements are equivalent:
\begin{enumerate}
 \item[(i)] A $C^*$-algebra $B$ is $D$QWEP$_1$;
 \item[(ii)] For any ucp map $u: C^*\F_{\infty} \to B$, the map $u \otimes \id$ extends to a continuous map from $C^*\F_{\infty}\otimes_{max_1^D}C^*\F_{\infty}$ to $B \otimes_{max}C^*\F_{\infty}$, where $max_1^D$ is the induced norm from the inclusion $C^*\F_{\infty}\otimes C^*\F_{\infty}\subseteq\mathcal{L(H}^u_D)\otimes_{max}C^*\F_{\infty}$.
 \end{enumerate}
\end{theorem}

\begin{proof}
 (i)$\Rightarrow$(ii): Suppose $B$ is $D$QWEP$_1$. Then $B=A/J$ for some $C^*$-algebra $A$ with $D$WEP$_1$. Let $u: C^*\F_{\infty} \to B$ be a ucp map, and $\pi:A\to B$ be the quotient map. Since $C^*\F_{\infty}$ has the lifting property, there exists a ucp map $\varphi: C^*\F_{\infty} \to A$ which lifts $u$, i.e. the following diagram commutes
\[
\xymatrix{
 C^*\F_{\infty}\ar [r]^{u} \ar[d]_{\varphi} & B\\
    A\ar [ur]_{\pi} &
  }   
\]

 By Theorem ~\ref{tensor wep}, we have $A\otimes_{max_1^D} C^*\F_{\infty} = A\otimes_{max} C^*\F_{\infty}$. Therefore, we have the following continuous maps
 \[
  C^*\F_{\infty} \underset{max^D_1}\otimes C^*\F_{\infty} \overset{\varphi\otimes \id}\longrightarrow A \underset{max^D_1}\otimes C^*\F_{\infty} =  A \underset{max}\otimes C^*\F_{\infty}\overset{\pi\otimes \id}\longrightarrow B \underset{max}\otimes C^*\F_{\infty}.   
\]
Note that $(\pi\otimes \id)\circ(\varphi\otimes \id)|_{C^*\F_{\infty}\otimes 1_{C^*\F_{\infty}}}=u$ by the lifting property. Therefore, $u\otimes \id$ extends to a continuous map from $C^*\F_{\infty}\otimes_{max_1^D}C^*\F_{\infty}$ to $B \otimes_{max}C^*\F_{\infty}$.

(ii)$\Rightarrow$(i):  Let $u: C^*\F_{\infty} \to B$ be the quotient map. We have the following diagram
 \[
\xymatrix{
 C^*\F_{\infty} \underset{max_1^D}\otimes C^*\F_{\infty} \ar [r]^{\qquad{u\otimes id}} \ar@ {^{(}->}[d]&B\underset{max}\otimes C^*\F_{\infty} \ar [r] & \mathbb{B}(\HH)\\
    \mathcal{L(H}^u_D)\underset{max}\otimes C^*\F_{\infty} \ar [urr], &
  }   
\]
where $\BB(\HH)$ is the universal representation of $B$. By Arveson's extension theorem, there exists a ucp map $\Phi:\mathcal{L(H}^u_D)\otimes_{max} C^*\F_{\infty}\to\mathbb{B}(\HH)$. Using \emph{The Trick} (see proof of Theorem ~\ref{tensor wep}), we get a map $\phi: \mathcal{L(H}^u_D)\to B^{**}$. 
Let $\md(\phi)$ be the multiplicative domain of $\phi$. Note that $C^*\F_{\infty}\subset \md(\phi)$. Therefore, $\phi$ maps $\md(\phi)$ onto a $C^*$-subalgebra of $B^{**}$ containing $B$. Let $A=\md(\phi)\cap \phi^{-1}(B)$. Then by Corollary ~\ref{ki md}, $A$
is \emph{r.w.i.} in $\mathcal{L(H}^u_D)$, so $A$ has the $D$WEP$_1$. Hence $B$ as a quotient of $A$ is $D$QWEP$_1$.
\end{proof}

\begin{remark}
\normalfont
In the proof of the above Theorem, we showed that the second statement is equivalent to the statement that for any ucp map $u: C^*\F_{\infty} \to B$, $w:C^*\F_{\infty} \to B^{\op}$, the map
 $u \otimes w$ extends to a continuous map from $C^*\F_{\infty}\otimes_{max_1^D}C^*\F_{\infty}$ to $B \otimes_{max} B^{\op}$.
\end{remark}

Now let us investigate some basic properties of the $D$QWEP. We have the following proposition similar to the $D$WEP case.
\begin{proposition}
\label{min qwep}
The following hold:

\begin{enumerate}\item If a $C^*$-algebra $B$ is $D$QWEP$_1$ and $C$ is nuclear, then $C\otimes_{min}B$ is also $D$QWEP$_1$.
 \item If von Neumann algebras $M$ and $N$ are $C$QWEP$_2$ and $D$QWEP$_2$, respectively, then $M \bar{\otimes} N$ is $(C\otimes_{min}D)$QWEP$_2$. 
 \end{enumerate}
\end{proposition}

\begin{proof}
(1) Suppose $B$ is $D$QWEP$_1$, then $B=A/J$ for some C$^*$-algebra $A$ with the $D$WEP$_1$. Since $C$ is nuclear, it is also exact. Therefore, we have 
\[
C\underset{min}\otimes B=C\underset{min}\otimes (A/J)\cong \frac{C\otimes_{min} A}{{C\otimes_{min} J}}.
\]
But $C\otimes_{min} A$ has the $D$WEP$_1$ by Proposition ~\ref{min wep}(1). Therefore, $C\otimes_{min}B$ is $D$QWEP$_1$.

(2) Since $M$ is $C$QWEP$_2$, it is \emph{r.w.i.} in $\mathcal{L}^w(\HH_{C^{**}})^{**}$. Similarly, $N$ is \emph{r.w.i.} in $\mathcal{L}^w(\HH_{D^{**}})^{**}$. Therefore, we have ucp maps 
\[
M\bar{\otimes}N \overset{r.w.i.}\hookrightarrow \mathcal{L}^w(\HH_{C^{**}})^{**} \bar{\otimes} \mathcal{L}^w(\HH_{D^{**}})^{**} \overset{r.w.i.}\hookrightarrow \mathcal{L}^w(\HH_{C^{**}\bar{\otimes}D^{**}})^{**}.
\]\
Note that by the same argument as in the proof of Proposition ~\ref{min wep} (2), $\mathcal{L}^w(\HH_{C^{**}\bar\otimes D^{**}})^{**}$ is \emph{r.w.i.} in $\mathcal{L}^w(\HH_{(C\otimes_{min} D)^{**}})^{**}$. Hence $M \bar{\otimes} N$ is \emph{r.w.i.} in $\mathcal{L}^w(\HH_{(C\otimes_{min} D)^{**}})^{**}$. Therefore, $M \bar{\otimes} N$ is $(C\otimes_{min}D)$QWEP$_2$.
\end{proof}

By Theorem ~\ref{12}, $D$WEP$_1$ implies $D$WEP$_2$, and hence $D$QWEP$_1$ implies $D$QWEP$_2$. In Section 5 we will show that there exist $C^*$-algebras with $D$WEP$_2$ which do not have $D$WEP$_1$. However in the QWEP context, the two concepts coincide. 
To see this, we need the following lemmas in which we use Kirchberg's categorical method. 

\begin{remark}
\label{wepqwep}
\normalfont
If a C$^*$-algebra $A$ has the $D$WEP$_2$, then it is $D^{**}$QWEP$_1$. Indeed since $A$ has the $D$WEP$_2$, it is \emph{r.w.i.} in $\BB(\ell_2)\bar\otimes D^{**}=(\mathcal{K}\otimes D)^{**}$. Now since $D$ is $D^{**}$QWEP$_1$, so is $\mathcal{K}\otimes D$ and therefore, so is $(\mathcal{K}\otimes D)^{**}$. Hence $A$ is $D^{**}$QWEP$_1$.
\end{remark}

The next lemma shows that $D$QWEP$_i$, for $i=1,2$, is stable under the direct products.

\begin{lemma}
\label{prod qwep}
Suppose $(B_i)_{i\in I}$ is a net of $C^*$-algebras in $\BB(\mathcal{H})$. If $B_i$ is $D$QWEP$_i$, for all $i\in I$, then so is $\Pi_{i\in I} B_i$.
\end{lemma}

\begin{proof}
Since $B_i$ is $D$QWEP$_i$, it is a quotient of a $C^*$-algebra $A_i$ with $D$WEP$_i$. By Lemma ~\ref{prod wep}, $\Pi_{i\in I} A_i$ has the $D$WEP$_i$. Therefore, $\Pi_{i\in I} B_i$ is $D$QWEP$_i$.
\end{proof}

\begin{lemma} 
\label{wri qwep}
Let $B$ be a $D$QWEP$_i$ C$^*$-algebra, for $i=1,2$, and $B_0$ a C$^*$-subalgebra of $B$ which is relatively weakly injective in $B$. Then $B_0$ is also a $D$QWEP$_i$ C$^*$-algebra.
\end{lemma}

\begin{proof}
If $B$ is $D$QWEP$_i$, then it is a quotient of a $C^*$-algebra $A$ with $D$WEP$_i$. Let $\pi:A\to B$ be the quotient map, $B=A/J$ and $A_0=\pi^{-1}(B_0)$. Then $A_0$ is \emph{r.w.i.} in $A$. In fact this follows from the fact that
\[
 A_0^{**}=J^{**}\oplus B_0^{**}\subset J^{**}\oplus B^{**}=A^{**}.
\]
Now by Lemma ~\ref{wri wep}, 
$A_0=\pi^{-1}(B_0)$ has the $D$WEP$_i$. Hence $B_0$ is $D$QWEP$_i$.
\end{proof}

\begin{lemma}
\label{uball}
Let $A$ and $B$ be unital C$^*$-algebras. Suppose there exists a map $\psi :A\to B$ which maps the closed unit ball of $A$ onto the closed unit ball of $B$. If $A$ has the $D$WEP$_i$, then $B$ is $D$QWEP$_i$, for $i=1,2$. 
\end{lemma}

\begin{proof}
Let $A_0 \subset A$ be the multiplicative domain of $\psi$. Since $\psi$ maps the closed unit ball of $A$ onto that of $B$, the restriction of $\psi$ on $A_0$ is a surjective $*$-homomorphism onto $B$. Let $\pi=\psi|_{A_0}$.

By Lemma ~\ref{ki md}, we have $A_0$ is \emph{r.w.i.} in $A$ and hence it has the $D$WEP$_i$ by Lemma ~\ref{wri wep}. Since $B$ is a quotient of $A_0$, $B$ is $D$-QWEP$_i$. 
\end{proof}

\begin{corollary}
\label{quball}
Let $B$ and $C$ be $C^*$-algebras. Suppose $B$ is $D$QWEP$_i$, and $\psi: B\to C$ is a ucp map that maps the closed unit ball of $B$ onto that of $C$. Then $C$ is $D$QWEP$_i$.
\end{corollary}

\begin{proof}
 Since $B$ is $D$QWEP$_i$, there exists a $C^*$-algebra $A$ with the $D$WEP$_i$, and a surjective $^*$-homomorphism $\pi: A \to B$. 
Notice that $\pi$ maps closed unit ball of $A$ onto that of $B$. Hence the composition $\psi \circ \pi$ maps the closed unit ball of $A$ onto that of $C$. By Lemma ~\ref{uball}, $C$ is $D$QWEP$_i$.
\end{proof}

\begin{lemma} 
\label{dc}
Suppose $(B_i)_{i\in I}$ is an increasing net of $C^*$-algebras in $\BB(\mathcal{H})$. If all $B_i$ are $D$QWEP$_i$, then $\overline{\cup B _i}$ and $(\cup B_i)''$ are $D$QWEP$_i$.
\end{lemma}

\begin{proof}
Let $B=\cup B_i$. It suffices to show that $B''$ is $D$QWEP$_i$. Since $B_i$ is $D$QWEP$_i$, there exists a $C^*$-algebra $A_i$ with $D$WEP$_i$, and a surjective $^*$-homomorphism $\pi_i: A_i\to B_i$.
Let $J$ be a directed set containing $I$. By Lemma \ref{prod wep}, $\prod_{j\in J} A_j$ has the $D$WEP$_i$. Fix a free ultrafilter $\mathcal{U}$ on the net $J$. Define a ucp map $\varphi: \prod_{j\in J} A_j\to B''$ by $\varphi((x_j)_{j\in J}) = \lim_{j \to \mathcal{U}} \pi(x_j)$ in the ultraweak topology. By Kaplansky's density theorem, if $J$ is large enough, then $\varphi$ maps the closed unit ball of $\prod_{j\in J} A_j$ onto that of $B''$. Now by Lemma \ref{uball}, $B''$ is $D$QWEP$_i$.
\end{proof}

The next corollary shows that unlike the $D$WEP case, the $D$QWEP of a C$^*$-algebra and its double dual are equivalent.
\begin{corollary}
\label{ds}
A $C^*$-algebra $B$ is $D$QWEP$_i$ if and only if $B^{**}$ is $D$QWEP$_i$ for $i=1,2$.
\end{corollary}

\begin{proof}
The ``if'' direction follows directly from Lemma ~\ref{wri qwep} since $B$ is \emph{r.w.i.} in $B^{**}$. For the other direction, we can apply Lemma ~\ref{dc} to $B$ together with its universal representation.
\end{proof}

\begin{lemma}
\label{factor qwep}
Suppose $B$ and $C$ are $C^*$-algebras, and $B$ factors through $C$ approximately via ucp maps in the point-weak$^*$ topology. 
If $C$ is $D$QWEP$_i$, then so is $B$.
\end{lemma}

\begin{proof}
Since $B$ factors through $C$, there are families of ucp maps $\alpha_i: B\to C$ and $\beta_i: C\to B$, $i\in I$ such that $\beta_i\circ \alpha_i$ converges to the identity map on $B$ in the point-weak$^*$ topology, i.e.
\[
\lim_{x,\mathcal{U}} (\beta_i \circ \alpha_i)(x)(x^*)= x^*(x)
\]
for $x\in B$, $x^*\in B^*$ and an ultrafilter $\mathcal{U}$.
Define $\alpha: B\to \prod_{i\in I} C$ by $\alpha(x)=(\alpha_i(x))_{i\in I}$, for $x\in B$. Let $\beta: \prod_{i\in I}C\to B^{**}$, $\beta=\lim_{i\to \mathcal{U}}\beta_i$. Define $\beta^{\#}:B^*\to \prod _{\mathcal{U}}C^*$, 
by  $\beta^{\#}(x^*)=(\beta^*(x^*))$. In fact $\beta^{\#}=\beta^{*}|_{B^*}$. Now the following map gives the identity on $B$:

\[
\xymatrix{
B\ar [r]^{\alpha} &\prod C\ar [r]^{(\beta^{\#})^*}& B^{**}.
}
\]
Taking the duals, we have
\[
\xymatrix{
B^*\ar [r]^{\beta^{\#}} &(\prod_{\mathcal{U}} C)^*\ar [r]^{\alpha^*}& B^{*}.
}
\]

This gives a conditional expectation from $C^{**}$ to $B^{**}$ which is identity on $B$. Therefore, $B$ is \emph{r.w.i.} in $C^{**}$. By Corollary ~\ref{ds}, $C^{**}$ is $D$QWEP$_i$. Hence so is $B$.
\end{proof}

\begin{corollary}
If a $C^*$-algebra $B$ is $D$QWEP$_i$, for $i=1,2$, then so is $\mathcal{M}(B)$.
\end{corollary}

\begin{proof}
Note that by Theorem ~\ref{factor1}, the identity map on $\M(B)$ factors through $\ell_{\infty}(B)$ approximately via ucp maps in point weak $^*$-topology. Since $B$ is $D$QWEP$_i$, by  Lemma ~\ref{prod qwep}, so is $\ell_{\infty}(B)$. Therefore, by Lemma ~\ref{factor qwep}, $\M(B)$ is $D$QWEP$_i$.
\end{proof}

We have the following transitivity result for $D$QWEP$_i$. We only show the $D$QWEP$_1$ case. The proof of the other case is similar. First we need the following lemma.

\begin{lemma}
\label{module}
Let $D$ be a $C^*$-algebra. If $D$ is $C$QWEP$_i$ for $i=1,2$, then so are $\mathcal{L(H}_D)$ and $\mathcal{L}^w(\HH_{D^{**}})$.
\end{lemma}

\begin{proof}
 If $D$ is $C$QWEP$_1$, then by Proposition ~\ref{min qwep}(1), so is $\mathcal{K}\otimes D$. By Theorem ~\ref{factor1}, $\mathcal{L(H}_D)=\mathcal{M}(\mathcal{K}\otimes D)$ factors through $\ell_{\infty}(\mathcal{K}\otimes D)$, and therefore, it is $D$QWEP$_1$, by Lemma ~\ref{factor qwep}. Hence it is also $D$QWEP$_2$.
 For the other case, it suffices to show that $\BB(\HH)\bar\otimes D^{**}$ is $D$QWEP$_1$. Note that $\BB(\HH)\bar\otimes D^{**}=(\K\otimes D)^{**}$ and $\K\otimes D$ is $D$QWEP$_1$. By Corollary ~\ref{ds}, $(\K\otimes D)^{**}$ is $D$QWEP$_1$, and hence it is $D$QWEP$_2$. 
\end{proof}

The following result shows the transitivity of the $D$QWEP$_i$ for $i=1,2$.

\begin{corollary}
Let $B$, $C$ and $D$ be $C^*$-algebras such that $B$ is $D$QWEP$_i$, and $D$ is $C$QWEP$_i$. Then $B$ is $C$QWEP$_i$.
\end{corollary}

\begin{proof}
We only show this for $i=1$. Let $\CF\subset \LL(\HH^u_D)$ be the universal representation. Since $B$ is $D$QWEP$_1$, by Theorem ~\ref{tensor qwep},  for all ucp maps $u: C^*\F_{\infty}\to B$, the map 
$u\otimes \id: C^*\F_{\infty}\otimes_{max_1^D} C^*\F_{\infty}\to B\otimes_{max}C^*\F_{\infty}$ is continuous, where $max_1^D$ is the norm induced from the inclusion 
$C^*\F_{\infty}\otimes C^*\F_{\infty}\subset \mathcal{L(H}^u_D)\otimes_{max} C^*\F_{\infty}$.
Since $D$ is $C$QWEP$_1$, so is $\mathcal{L(H}^u_D)$ by Lemma ~\ref{module}. Now by the tensor characterization of $C$QWEP$_1$, the map $w\otimes id: C^*\F_{\infty}\otimes_{max_1^C}C^*\F_{\infty}\to \mathcal{L(H}^u_D)\otimes_{max}C^*\F_{\infty}$ 
is continuous for all ucp maps $w:C^*\F_{\infty}\to\mathcal{L(H}^u_D)$. Now let $w$ be a faithful representation $C^*\F_{\infty}\to\mathcal{L(H}^u_D)$. Then we have following diagram
 \[
\xymatrix{
   C^*\F_{\infty} \underset{max_1^D}\otimes C^*\F_{\infty} \ar@ {^{(}->}[r] \ar [d]_{u\otimes \id}& \mathcal{L(H}^u_D) \underset{max}\otimes C^*\F_{\infty}\\
    B \underset{max}\otimes C^*\F_{\infty} & C^*\F_{\infty} \underset{max_1^C}\otimes C^*\F_{\infty}\ar[u]_{w\otimes \id} \ar@ {^{-}->}[ul] 
  }   
\]
Note that the image of $w\otimes \id$ is $C^*\F_{\infty} \otimes_{max_1^D} C^*\F_{\infty}$. Therefore, we get a continuous map from $C^*\F_{\infty} \otimes_{max_1^C} C^*\F_{\infty}$
to $B \otimes_{max} C^*\F_{\infty}$. This proves that $B$ is $C$QWEP$_1$.
\end{proof}

Now we are ready to establish the equivalence between the $D$QWEP notions by observing the following result.

\begin{theorem}
 \label{equi}
For a $C^*$-algebra $B$, the following conditions are equivalent:
\begin{enumerate}
 \item $B$ is $D$QWEP$_1$;
 \item $B$ is $D$QWEP$_2$;
 \item $B^{**}$ is $D^{**}$QWEP$_1$;
 \item $B^{**}$ is $D^{**}$QWEP$_2$.
\end{enumerate}
\end{theorem}

\begin{proof}
 (1)$\Rightarrow$(2): This follows from the fact that $D$WEP$_1$ implies $D$WEP$_2$. 

 (2)$\Rightarrow$(3): Suppose $B$ is $D$QWEP$_2$. Therefore, $B$ is the quotient of a $C^*$-algebra $A$ which is \emph{r.w.i.} in $\mathcal{L}^w(E_{D^{**}})$. By Remark \ref{wepqwep}, since $\mathcal{L}^w(E_{D^{**}})$ has the $D^{**}$WEP$_1$, it is $D^{**}$QWEP$_1$. Hence $A$ is $D^{**}$QWEP$_1$, and therefore, $B$ is $D^{**}$QWEP$_1$.

 (3)$\Rightarrow$(4): Follows from (1)$\Rightarrow$(2).

 (4)$\Rightarrow$(1):  Suppose $B^{**}$ is $D^{**}$QWEP$_2$, and therefore so is $B$ by Corollary ~\ref{ds}. Then $B$ is the quotient of a $C^*$-algebra $A$ which is \emph{r.w.i.} in $\mathcal{L}^w(E_{D^{****}})$. We have 
 \[
 A \overset{r.w.i.}\subset\mathcal{L}^w(E_{D^{****}})\overset{r.w.i.}\subset\mathbb{B}(\ell_2)\bar{\otimes} D^{****}=(\mathcal{K}\underset{min}\otimes D^{**})^{**}.
 \]
Therefore, it suffices to show that $\mathcal{K}\otimes_{min} D^{**}$ is $D$QWEP$_1$. Notice that $\mathcal{K}\otimes_{min} D^{**}$ factors through $\prod_n M_n(D^{**})$ approximately via ucp maps in point-norm topology, since $\cup M_n(D^{**})$ is norm-dense in $\mathcal{K}\otimes_{min} D^{**}$. 
Now since $D$ has the $D$WEP$_1$, $D^{**}$ is $D$QWEP$_1$. Therefore, by Proposition ~\ref{min qwep}, so is $M_n(D^{**})=M_n\otimes_{min}D^{**}$. Hence by Lemma ~\ref{factor qwep}, $\mathcal{K}\otimes_{min} D^{**}$ is $D$QWEP$_1$. This finishes the proof.
 \end{proof}

\section{Illustrations}

In Section 3, we showed that $D$WEP$_1$ implies $D$WEP$_2$. Our first example will show the converse is not true, 
and hence the two notions of $D$WEP are not equivalent. 

\begin{example} 
\normalfont
\label{Bl2} Let $D = \BB(\ell_2)$. Note that $\mathcal{L(H}_D) = \M(\K\otimes \BB(\ell_2))$, and $\K\otimes \BB(\ell_2)$ has the WEP, 
and so does $\M(\K\otimes \BB(\ell_2))$. Therefore the two notions of $D$WEP$_1$ 
and WEP coincide. On the other hand, the $D$WEP$_2$ of a $C^*$-algebra is the same as being \emph{r.w.i.} in $\BB(\HH)\bar{\otimes} {\BB(\ell_2)}^{**}$. Notice that 
$\BB(\mathcal{H})\bar{\otimes} \BB(\ell_2)^{**}=(\mathcal{K} \otimes \BB(\ell_2))^{**}$ is QWEP. Therefore by Proposition ~\ref{bl2qwep}, 
$D$WEP$_2$ is equivalent to QWEP. 
Hence if $A$ is a QWEP $C^*$-algebra without the WEP, for instance $C^*_{\lambda}\F_n$, then $A$ has the $D$WEP$_2$ but not the $D$WEP$_1$, 
for $D = \BB(\ell_2)$.
\end{example}

Now we are ready to see some examples of relative WEP and QWEP over special classes of C$^*$-algebras. 

\begin{proposition}
\label{ncl}
 Let $D$ be a nuclear $C^*$-algebra. Then a $C^*$-algebra $A$ has the $D$WEP$_i$ for $i=1,2$ if and only if it has the WEP. 
\end{proposition}

\begin{proof}
 Suppose $A$ has the WEP. Therefore $A$ has the $D$WEP$_1$, and hence the $D$WEP$_2$. 

Now assume $A$ has the $D$WEP$_2$, i.e. it is \emph{r.w.i.} in $\BB(\ell_2)\bar{\otimes} D^{**}$.
Since $D$ is nuclear, $D^{**}$ is injective. Hence we have $D^{**} \subseteq \BB(\mathcal{H})\overset{\E}\to D^{**}$, where $\E$ is a conditional expectation. 
Let $CB(A,B)$ be the space of completely bounded maps from $A$ to $B$. Therefore we have
\[
 CB(S_1, D^{**})\overset{\pi} \hookrightarrow CB(S_1, \BB(\mathcal{H}))\overset{\varphi}\to CB(S_1, D^{**}),
\]
where $S_1$ is the algebra of trace class operators, $\pi$ is a $^*$-homomorphism, and $\varphi$ acts by composing the maps in $CB(S_1, \BB(\mathcal{H}))$ and $\E$. 
Note that by operator space theory $CB(S_1, D^{**})\simeq\BB(\ell_2)\bar{\otimes} D^{**}$ and $CB(S_1, \BB(\mathcal{H}))\simeq\BB(\ell_2) \bar{\otimes} \BB(\mathcal{H})=\BB(\ell_2 \otimes \mathcal{H})$. 
Hence we have the maps $\BB(\ell_2)\bar{\otimes} D^{**}\overset{\pi}\to \BB(\ell_2)\bar{\otimes}\BB(\mathcal{H})= \BB(\ell_2 \otimes \mathcal{H})\overset{\varphi}\to \BB(\ell_2)\bar{\otimes} D^{**}$. Now by Remark ~\ref{unitization}
we can unitize these two maps. Therefore $A$ is \emph{r.w.i.} in $\BB(\ell_2\otimes\mathcal{H})$, and hence it has the WEP.
\end{proof}

After nuclear C$^*$-algebras, it is natural to consider the relative WEP for an exact $C^*$-algebra $D$. For convenience, we consider the following stronger version of weak exactness property. 
A von Neumann algebra $M\subseteq \BB(\mathcal{H})$ is said to be \emph{algebraically weakly exact}, (a.w.e. for short), if there exists a weakly dense exact $C^*$-algebra $D$ in $M$. 
By \cite{Ki2}, we know that the a.w.e. implies the weak exactness. 

Notice that the unitization trick works better in $\mathcal{C}_2$ category, and hence we have the following.

\begin{proposition}
\label{exact}
 A $C^*$-algebra has the $D$WEP$_2$ for some exact $C^*$-algebra $D$ if and only if it is relatively weakly injective in an a.w.e. von Neumann 
algebra.
\end{proposition}

\begin{proof}
 Suppose a $C^*$-algebra $A$ has the $D$WEP$_2$, then $A$ is \emph{r.w.i.} in $\BB(\mathcal{H})\bar{\otimes}D^{**}$. Since both $\K$ and $D$ are 
exact C$^*$-algebras, so is $\K\otimes D$. Note that $\K\otimes D$ is weakly dense in $(\K\otimes D)^{**}=\BB(\mathcal{H})\bar{\otimes}D^{**}$. 
We have $\BB(\mathcal{H})\bar{\otimes}D^{**}$ is a.w.e.

For the other direction, suppose $A$ is \emph{r.w.i.} in an a.w.e von Neumann algebra $M$. Let $D$  be an exact $C^*$-algebra with $D''=M$.
Then there exists a central projection $z$ in $D^{**}$ such that $M=zD^{**}$. Hence we have completely positive maps 
$M \hookrightarrow D^{**} \to M$, which preserves the identity on $M$. Therefore by unitization $M$ is \emph{r.w.i.} in $\BB(\mathcal{H})\bar{\otimes}D^{**}$ for some infinite dimensional Hilbert space $\HH$. 
Hence if $A$ is \emph{r.w.i.} in $M$, 
then it is also \emph{r.w.i.} in $\BB(\mathcal{H})\bar{\otimes}D^{**}$, and therefore it has the $D$WEP$_2$.
\end{proof}

As we see, the nuclear-WEP is equivalent to the WEP. But the exact-WEP is different. 

\begin{example}
\normalfont
Let $\F_2$ be the free group of two generators. Then it is exact and hence $C^*_{\lambda}\F_2$ is exact and $L\F_2$ is weakly exact. Since 
$C^*_{\lambda}\F_2$ is \emph{r.w.i.} in $L\F_2$, by Proposition ~\ref{exact}, $C^*_{\lambda}\F_2$ has the $D$WEP$_2$ for $D=C^*_{\lambda}\F_2$. 
But $C^*_{\lambda}\F_2$ does not have the WEP, since the WEP of a reduced group C$^*$-algebra is equivalent to the amenability of the group (see Proposition 3.6.9 in \cite{BrOz}).
\end{example}

Now we consider the full group $C^*$-algebra of free group $C^*\F_\infty$. Since it is universal in the sense that for any unital separable 
$C^*$-algebra $A$, we have a quotient map $q: C^*\F_\infty \to A$. By the unitization trick, we have the following.

\begin{proposition}
 Let $A$ be a unital separable $C^*$-algebra. Then it has the $D$WEP$_2$ for $D=C^*\F_\infty$.
\end{proposition}

\begin{proof}
 Since we have a quotient map $q: C^*\F_\infty \to A$, there exists a central projection $z$ in ${C^*\F_\infty}^{**}$ such that $A^{**}=z{C^*\F_\infty}^{**}$. 
Hence we have an embedding $A^{**}\hookrightarrow \BB(\mathcal{H})\bar{\otimes}{C^*\F_\infty}^{**}$ with a completely positive map from $\BB(\mathcal{H})\bar{\otimes}{C^*\F_\infty}^{**}$ to $A^{**}$ 
by multiplying $1\otimes z$. By the unitization trick in Remark ~\ref{unitization}, $A^{**}$ has the $D$WEP$_2$ for $D=C^*\F_\infty$ and so does 
$A$, since $A$ is \emph{r.w.i.} in $A^{**}$. 
\end{proof}

It is natural and even more interesting to ask whether the full group C$^*$-algebra $C^*\F_\infty$ has $D$WEP, for $D$ is the reduced group 
C$^*$-algebra $C^*_{\lambda}\F_2$. In fact, this is related to the QWEP conjecture. If $\CF$ has the $D$WEP$_1$ for some WEP algebra $D$, then it has the WEP by Corollary ~\ref{trand} of transitivity. 
If $\CF$ does not have the $D$WEP$_1$ for some C$^*$-algebra $D$, then it does not have the WEP either.  
At the time of writing, we do not have an answer for this question. 

Now let us discuss some properties of being module QWEP relative to some special classes of C$^*$-algebras. In the rest of this section, we will examine the relation between one of the equivalent statements 
of Theorem ~\ref{equi} (for example statement (1), $B$ is $D$QWEP$_i$), and the statement that $B^{**}$ is $D^{**}$WEP$_i$, for either $i=1$ or $2$.

\begin{proposition}
 Let $B$ be a C$^*$-algebra. If $B^{**}$ has the $D^{**}$WEP$_i$, then $B$ is $D$QWEP$_i$, for $i=1,2$.
\end{proposition}

\begin{proof}
 Suppose $B^{**}$ has the $D^{**}$WEP$_i$, and hence $B^{**}$ is $D^{**}$QWEP$_i$ by the trivial quotient. By Theorem \ref{equi}, $B$ is $D$QWEP$_i$. 
\end{proof}

For some C$^*$-algebra $D$, the four equivalent statements in Theorem \ref{equi} are equivalent to the statement that $B^{**}$ has the $D^{**}$WEP$_i$. 
But this is not true in general. We will show examples of both circumstances. 

\begin{example}
\normalfont
Let $D=\BB(\ell_2)$. Then a $C^*$-algebra $B$ is $D$QWEP$_i$ if and only if $B^{**}$ has the $D^{**}$WEP$_i$, 
since they are both equivalent to $B$ being QWEP. Indeed, if $B$ is $D$QWEP$_1$, then $B=A/J$
and $A$ has the $D$WEP$_1$. Since $\mathcal{L(H}_D)$ has the WEP as shown in Example \ref{Bl2}, so does $A$, and hence $B$ is 
QWEP. On the other hand, having
$\BB(\ell_2)^{**}$WEP$_1$ is equivalent to being \emph{r.w.i.} in $\M(\K\otimes \BB(\ell_2)^{**})$, which is QWEP. Hence $B^{**}$ is QWEP. 
By Proposition ~\ref{dd}, $B$ is QWEP as well. 
\end{example}

\begin{example}
\normalfont
\label{nclq}
 Let $D$ be a nuclear C$^*$-algebra. Then the above statements are not equivalent. 
Indeed, it follows from Proposition ~\ref{ncl} that a C$^*$-algebra is $D$QWEP$_i$ if and only if it is QWEP. On the other hand, 
assume that $B^{**}$ has the $D^{**}$WEP$_1$. Note that 
$D^{**}$WEP$_1$ implies $D$WEP$_2$ by Remark ~\ref{vn wep}, which is equivalent to WEP by Proposition ~\ref{ncl}, and $B^{**}$ has the WEP if and only if it is injective. Therefore 
the fact that a C$^*$-algebra $B$ is $D$QWEP$_i$ does not imply that $B^{**}$ has the $D^{**}$WEP$_1$. 
\end{example}

\begin{example}
\normalfont
 For a von Neumann algebra $M$, let us compare the properties $M$QWEP$_1$ of $B$ and the $M^{**}$WEP$_1$ of $B^{**}$. We have the following 
partial results.

Case (i): $M$ is of type I$_n$. Then $M$ is subhomogeneous, which is equivalent to nuclearity. By Example ~\ref{nclq}, these two statements are not equivalent.  

Case (ii): $M$ is of type I$_{\infty}$, then $\BB(\ell_2)\bar{\otimes}M$ is \emph{r.w.i.} in $M$. Suppose $B$ is $M$QWEP$_1$, then 
$B$ is a quotient of a $C^*$-algebra $A$ which is \emph{r.w.i.} in $\BB(\ell_2)\bar{\otimes}M$. Hence $B^{**}$ is \emph{r.w.i.} in $A^{**}$ and hence in 
$(\BB(\ell_2)\bar{\otimes}M)^{**}$, and hence in $M^{**}$. Since $M^{**}$ is isomorphic to $\LL(\HH_{M^{**}})$ for 1-dimensional Hilbert 
space $\mathcal{H}$, it follows that $B^{**}$ has the $M^{**}$WEP$_1$.

Case (iii): $M$ is of type II$_\infty$ or III, then $\BB(\ell_2)\bar{\otimes}M \simeq M$. By a similar argument to that of Case (ii), we have the same conclusion.

Case (iv): $M$ is of type II$_1$ and a McDuff factor, i.e. $M\bar{\otimes}R \simeq M$. Then we have
\[
 M\simeq M\bar{\otimes}R \simeq M\bar{\otimes}R\bar{\otimes}R \supseteq M\bar{\otimes}R \bar{\otimes} L_{\infty}[0,1] 
\supseteq M\bar{\otimes} \prod_{n=1}^{\infty} M_n \supseteq M\bar{\otimes} \BB(\ell_2)
\]
with conditional expectation from the larger algebra to the smaller for each inclusion. Hence $M\bar{\otimes} \BB(\ell_2)$ is \emph{r.w.i.} in $M$. 
By the same argument above, the equivalence is established.

Problem: $M$ is a non-McDuff II$_1$ factor. At the time of writing, we do not have an affirmative answer for this case. 
\end{example}

\section{Application to C$^*$-norms}

In this section, we will discuss some application of our tensor norm $max^D$ constructed in Section 3, to C$^*$-norms. 
We will follow the approach in \cite{OP14} to construct norms on $A\otimes B$ for C$^*$-algebras $A$ and $B$. 

Let $E$ be a $n$-dimensional subspace in $B$, and $C^*\langle E\rangle$ be the separable unital C$^*$-subalgeba of $B$ 
generated by $E$, which contains $E$ completely isometrically. For free group of countably infinite generators $\F_\infty$, we have a quotient map 
$\CF \to B$. Let $q: C^*\langle E\rangle \ast C^*\F_\infty \to B$ denote the free product of the inclusion $C^*\langle E\rangle \subset 
B$ and the quotient map $C^*\F_\infty \to B$, and let $I=\ker(q)$, so that we have $B \simeq (\CE \ast \CF) /I$. Following \cite{OP14}, 
let 
\begin{align}\label{enorm}
 A\underset{E}\otimes B = \displaystyle\frac{A\otimes_{min} (\CE \ast \CF)}{A\otimes_{min} I}.
\end{align}

Similarly, we can construct a new norm using the $max^D_1$ norm defined in Section 3. Recall that for a universal inclusion 
$A\subset \LL(\HH^u_D)$, the $max^D_1$ norm is the induced tensor norm from the inclusion 
$A\otimes C \subset \LL(\HH^u_D) \otimes_{max} C$. Now we define 
\begin{align}\label{ednorm}
 A\underset{D,E}\otimes B = \displaystyle\frac{A\otimes_{max^D_1} (\CE \ast \CF)}{A\otimes_{max^D_1} I}.
\end{align}

By their constructions, it is easy to see the following continuous inclusions
\[
 A\underset{min}\otimes B \supseteq A\underset{E}\otimes B \supseteq A\underset{D,E}\otimes B \supseteq A\underset{max}\otimes B. 
\]
The goal of this section is to determine the conditions which distinguish the above norms, and make the inclusions strict.

We will follow the notations in \cite{OP14}. Let us first recall the operator space duality $F^*\otimes_{min} E \subset CB(F, E)$ isometrically (see Theorem B.13 in \cite{BrOz}). 
This gives us a correspondence between a tensor $x = \sum_k f^*_k\otimes e_k \in F^*\otimes E$, and a map 
$\varphi_x: F \to E$ given by $\varphi_x(f) = \sum_k f^*_k (f)e_k$, with $\|x\|_{\min} = \|\varphi_x\|_{\cb}$. 
For finite dimensional operator space $E$, we denote by $t_E$ the ``identity'' element in $E^* \otimes E$. 
Note that $\|t_E\|_{\min} = 1$ and that any norm of $t_E$ is independent of embeddings $E^*\hookrightarrow \BB(\ell_2)$ and 
$E\hookrightarrow \BB(\ell_2)$. 

For any $n\in \mathbb{N}$, let $\OS$ denote the metric space of all $n$-dimensional operator spaces, equipped with the 
completely bounded Banach-Mazur distance. Note that by \cite{JP95}, $\OS$ is non-separable for $n\geq 3$. If $A$ is a separable C$^*$-algebra, 
then the set $\OS(A)$ of all $n$-dimensional operator subspaces of $A$ is a separable subset of $\OS$. 

The first lemma will help us distinguish $\|\cdot\|_{E,D}$ and $\|\cdot\|_{\min}$.  
\begin{lemma}\label{cbfact}
 Let $E$ and $F$ be subspaces of C$^*$-algebra $B$, and $E^*$, $F^*$ be subspaces of $C^*$-algebra $A$. 
Then $\|t_F\|_{E,D} \geq d_{cb}(F, \OS(\CF))$, where $d_{cb}(F, \OS(\CF)) = \inf \{d_{cb}(F, G)\text{ }|\text{ } G\in\OS(\CF)\}$.
\end{lemma}

\begin{proof}
 By their construction, we have the following diagram
 \[
\xymatrix{
A\underset{max^D_1}\otimes (\CE\ast\CF) \ar [d]_{q} \ar@ {^{(}->}[r] &  \LL(\HH^u_D) \underset{max}\otimes (\CE\ast\CF) & \LL(\HH^u_D) \underset{max}\otimes \CF\ar [l]_{\qquad{\pi}}\ar@ {^{(}->}[d]^{\iota} \ar@ {^{(}->}[d]^{\iota}  \\
A\underset{E,D}\otimes  B && \LL(\HH^u_D) \underset{min}\otimes \CF
}
\]

where $\pi$ is induced from a quotient map $\CF \to \CE\ast \CF$, and $\iota$ is a continuous inclusion. 

Note that for finite dimensional $F^* \subset A$ and $F\subset B$, we can lift $F$ to a subspace $G \subset \CE\ast\CF$, and then a 
$\tilde{G}\subset \CF$. Therefore the identity map $t_F$ on $F^*\otimes F \subset A\otimes_{E,D} B$ admits a lifting 
$\xi \in F^*\otimes \tilde{G}\subset \LL(\HH^u_D) \otimes_{min} \CF$ which corresponds to a map $\alpha: F\to \tilde{G}$. Hence we have a factorization
\[
 F \overset{\alpha} \longrightarrow \tilde{G} \overset{q\circ\pi}\longrightarrow B,
\]
such that the composition is the inclusion $F\subset B$. Therefore the image $\alpha(F)$ in $\tilde{G}$ is isomorphic to $F$. Hence we have
\[
\xymatrix{
 t_F: F\ar[r] & \alpha(F) \ar@ {^{(}->}[d]\ar[r] & F\\
 & \CF&
}
\]
and therefore, $\|t_F\|_{E,D} \geq d_{cb}(F, \OS(\CF))$.  
\end{proof}

The next lemma will help us distinguish $\|\cdot\|_{E,D}$ and $\|\cdot\|_{\max}$. 

\begin{lemma}\label{max}
 Let $A\subset \LL(\HH^u_D)$ be the universal representation of $A$. 
Also let $\pi$ be a surjective ucp from $B_1$ to $B$ with kernal $I$, such that 
$(A\otimes_{max^D_1} B_1) /(A\otimes_{max^D_1} I)   \simeq A\otimes_{max} B$.
If there exists a surjective completely positive map $\sigma: B \to A^{\op}$, 
then $A$ has the $D$WEP$_1$. 
\end{lemma}

\begin{proof}
 From the assumptions, we have the following diagram
  \[
\xymatrix{
 A \underset{max^D_1} \otimes B_1 \ar@ {^{(}->}[d]\ar [r]& A\underset{max}\otimes B \ar [r] & A\underset{max}\otimes A^{\op} \ar [r] & \BB(L_2(A^{**}))\\
 \LL(\HH^u_D)\otimes_{max} B_1\ar [urrr]
}
\]
By Arveson's extension theorem, there exists a ucp map $\Phi: \LL(\HH^u_D)\otimes_{max} B_1 \to \BB(L_2(A^{**}))$. 
Applying \emph{the Trick}, we obtain a ucp map $\phi: \LL(\HH^u_D) \to A^{**}$, which is identity on $A$, and hence 
$A$ has the $D$WEP$_1$. 
\end{proof}

Now we are ready to give the conditions which distinguish the norms. 

\begin{theorem}\label{norms}
 For the continuous inclusions 
\[
 A\underset{min}\otimes B \underset{(a)}\supseteq A\underset{E}\otimes B \underset{(b)}\supseteq A\underset{D,E}\otimes B \underset{(c)}\supseteq A\underset{max}\otimes B, 
\]
the strict inclusion holds for
\begin{enumerate}
 \item[(a),] if there exists $n$-dimensional subspaces $F^*\subset A$ and $F\subset B$, such that 
$F\not\in \OS(\CE\ast\CF)$;
 \item[(b),] if $E\not\in \OS(\CF)$;
 \item[(c),] if there exists a surjective ucp $B\to A^{\op}$, and $A$ does not have the $D$WEP$_1$.
\end{enumerate}
Moreover, for $n$-dimensional subspaces $E$, $F$ in $B$, and $E^*$, $F^*$ in $A$, 
we have $A\otimes_{D,E} B \neq A\otimes_F B$, if $F\not\in \OS(\CF)$. Therefore $A\otimes_{D,E} B$ gives us a new norm on $A\otimes B$, 
distinct from the continuum norms constructed in \cite{OP14}. 
\end{theorem}

\begin{proof}
 (a) is proved in \cite{OP14}. Indeed, if such $F$ and $F^*$ exist, then the identity map $t_F$ on 
$F^* \otimes F \subset A\otimes_{min} B$ has norm $1$. On the other hand, notice that the norm of
$t_F$ in $A\otimes_{E} B$ is greater than $1$. Indeed if $\|t_F\|_{E} = 1$, then by the construction of $A\otimes_{E} B$, 
it lifts to an element $\xi \in F^* \otimes (\CE\ast \CF)$ with $\|\xi\|_{\min} = 1$. This corresponds to 
a completely isometric mapping $F \to \CE\ast\CF$, showing that $F$ is completely isometric to a subspace of $\CE\ast\CF$, which contradicts the condition 
$F\not\in \OS(\CE\ast\CF)$. Hence $\|t_F\|_E > \|t_F\|_{\min} = 1$. 

 (b) By Lemma ~\ref{cbfact}, $\|t_E\|_{E,D} \geq d_{\cb}(E, \OS(\CF))$. If $E\not\in \OS(\CF)$, then we have $d_{\cb}(E, \OS(\CF)) > 1$, 
and so is $\|t_E\|_{E,D}$. Therefore $\|t_E\|_{E,D} > \|t_E\|_{E} = 1$. 

 (c) Apply Lemma ~\ref{max} to $B_1 = \CE\ast\CF$. Then by the construction we have 
$A\otimes_{D,E} B = (A\otimes_{max^D_1} B_1) /(A\otimes_{max^D_1} I)$. If $A\otimes_{D,E} B = A\otimes_{max} B$, 
then by Lemma ~\ref{max}, $A$ has the $D$WEP$_1$, which contradicts the condition. 

 Moreover, similar to the proof of (b), Lemma ~\ref{cbfact} shows that $A\otimes_{D,E} B \neq A\otimes_F B$, 
if $F\not\in \OS(\CF)$. 
\end{proof}

Now we will construct C$^*$-algebras $A$ and $B$, to give a concrete example with above distinct norms. Our goal 
is to construct a C$^*$-algebra $A$ such that $A\simeq A^{\op}$ without $D$WEP$_1$, and let $B = A$. 

Recall that for operator spaces $E$ and $F$, $\CE\ast C^*\langle F\rangle \simeq C^*\langle E\otimes_{h} F\rangle$, 
where $E\otimes_{h} F$ is the Haagerup tensor product, and also that  $C^*\langle E^{\op}\rangle \simeq \CE^{\op}$.

\begin{lemma}\label{haar}
 Let $C = C^*\langle E\otimes_{h} E^{\op}\rangle$. Then $C\simeq C^{\op}$. 
\end{lemma}

\begin{proof}
 Let $\pi: \CF \to \CE$ be the quotient map, then so is $\pi^{\op}: \CF^{\op} \to \CE^{\op}$. Then we have a quotient map 
$\CF \ast \CF^{\op} \to \CE \ast \CE^{\op}$, which maps the unitaries to unitaries. Notice that 
for $I$ the index set of $\F_{\infty}$, we have the following isomorphism given by
\begin{align*}
 \CF \ast \CF^{\op}  && \simeq && C^*\F_{I\times I}  &&\simeq && (C^*\F_{I\times I})^{\op} \\
  g_i \ast 1  &&\longmapsto && 1\times g_i^{-1}  &&\longmapsto && (1\times g_i)^{\op} \\
  1 \ast h_i  && \longmapsto && h_i^{-1}\times 1  &&\longmapsto && (h_i\times 1)^{\op} \\
\end{align*}

Let $\pi(g_i) = x$ and $\pi^{\op}(h_i) = y^{\op}$. Define the map $\CE \ast \CE^{\op} \to (\CE \ast \CE^{\op})^{\op}$, 
by $x\ast 1 \to (1\ast x^{\op})^{\op}$, and $1\ast y^{\op} \to (y\ast 1)^{\op}$. Then it is easy to check
that this is an isomorphism following from the isomorphism $\CF \ast \CF^{\op} \to (C^*\F_{I\times I})^{\op}$.
\end{proof}

Now we are ready to construct the example. For $n$-dimensional operator spaces $E$ and $F$ satisfying 
the conditions (a) and (b) in Theorem ~\ref{norms}, let 
\[
 D = C^*\langle (E\oplus E^* \oplus F \oplus F^*) \otimes_{h}(E\oplus E^* \oplus F \oplus F^*)^{\op}\rangle, 
\]
where the direct sum is in $\ell_{\infty}$. Then by Lemma \ref{haar}, $D\simeq D^{\op}$. 

Let $A = D \otimes_{min} C^*_{\lambda}\mathbb{F}_2$. Then we have
\[
A^{\op} = (D \underset{min}\otimes C^*_{\lambda}\mathbb{F}_2)^{\op} \simeq D^{\op}\underset{min}\otimes C^*_{\lambda}\mathbb{F}_2^{\op} \simeq 
D \underset{min}\otimes C^*_{\lambda}\mathbb{F}_2 = A.
\]
Let $B = A$, and hence we have a surjective ucp $B\to A^{\op}$. Also since $C^*_{\lambda}\mathbb{F}_2$ does not have the WEP, 
the faithful representation for $C^*_{\lambda}\mathbb{F}_2 \subset \BB(\HH)$ induces an inclusion 
$A = D\otimes_{min} C^*_{\lambda}\mathbb{F}_2\hookrightarrow D\otimes_{min} \BB(\HH)$ which is not \emph{r.w.i.}. Notice that this is not equivalent to $D$WEP. 
However if we construct the $max^D$ norm from the inclusion $A\subseteq D\otimes_{min} \BB(\HH)$ instead of $A\subseteq \LL(\HH^u_D)$, 
we will have the same conclusion that the four norms are distinct.  

\begin{corollary}
 Let $D$ be as above, and $A= D \otimes_{min} C^*_{\lambda}\mathbb{F}_2$. For a faithful representation $C^*_{\lambda}\mathbb{F}_2 \subset \BB(\HH)$, 
define the $max^D$ norm on $A\otimes C$ to be the tensor norm induced from the inclusion 
$A\otimes C \subseteq (D\otimes_{min} \BB(\HH)) \otimes_{max} C$. Let $B=A$. Define the quotient norms 
$A\otimes_E B$ as in ~(\ref{enorm}) and $A\otimes_{D,E}B$ as in ~(\ref{ednorm}) with the new $max^D$ norm as follows
\[
 A\underset{D,E}\otimes B = \displaystyle\frac{A\otimes_{max^D} (\CE \ast \CF)}{A\otimes_{max^D} I}.
\]
Then we have the following strict inclusions 
\[
 A\underset{min}\otimes B \supset A\underset{E}\otimes B \supset A\underset{D,E}\otimes B \supset A\underset{max}\otimes B. 
\]   
\end{corollary}

\begin{frame}

\end{frame}

\end{document}